\documentclass[reqno]{amsart}

\usepackage[usenames,dvipsnames]{pstricks}
\usepackage{epsfig}

\usepackage{color}
\date{\today}

\usepackage[latin1]{inputenc}
\usepackage[T1]{fontenc}
\usepackage{amsmath, amsthm}
\usepackage[english]{babel}
\usepackage{amssymb}
\usepackage{latexsym}
\usepackage{graphicx}
\usepackage[all]{xy}

\newtheorem{theorem}{Theorem}[section]

\newtheorem{proposition}[theorem]{Proposition}
\theoremstyle{definition}
\newtheorem{definition}[theorem]{Definition}
\newtheorem{example}[theorem]{Example}

\theoremstyle{remark}
\newtheorem{remark}[theorem]{Remark}

\newcommand{\ot}{\otimes}
\newcommand{\co}{\circ}

\begin{document}

\begin{center}

{\huge{\bf Weak Hopf Quasigroups}}

\end{center}

\ \\
\begin{center}
{\bf J.N. Alonso \'Alvarez$^{1}$, J.M. Fern\'andez Vilaboa$^{2}$, R.
Gonz\'{a}lez Rodr\'{\i}guez$^{3}$}
\end{center}

\ \\
\hspace{-0,5cm}$^{1}$ Departamento de Matem\'{a}ticas, Universidad
de Vigo, Campus Universitario Lagoas-Marcosende, E-36280 Vigo, Spain
(e-mail: jnalonso@ uvigo.es)
\ \\
\hspace{-0,5cm}$^{2}$ Departamento de \'Alxebra, Universidad de
Santiago de Compostela.  E-15771 Santiago de Compostela, Spain
(e-mail: josemanuel.fernandez@usc.es)
\ \\
\hspace{-0,5cm}$^{3}$ Departamento de Matem\'{a}tica Aplicada II,
Universidad de Vigo, Campus Universitario Lagoas-Marcosende, E-36310
Vigo, Spain (e-mail: rgon@dma.uvigo.es)
\ \\

{\bf Abstract} In this paper we introduce the notion of weak Hopf quasigroup 
  as a generalization of weak Hopf
algebras and Hopf quasigroups. We obtain its main properties and we prove the fundamental theorem of Hopf modules for these algebraic structures.

\vspace{0.5cm}



\section{introduction}

The notion of Hopf algebra and its generalizations appeared  as useful tools 
in relation with many branch of mathematics such that algebraic geometry, number theory, 
Lie theory, Galois theory, quantum group theory and so on.  A common principle to obtain generalizations of the original notion of Hopf algebra is to weak some of  axioms of its definition.  For example, if one does not force the coalgebra structure to respect the unit of the algebra structure, one is lead to weak Hopf algebras. In a different way, the weakening of the associativity leads to Hopf quasigroups and  quasi-Hopf  algebras. 

Weak Hopf algebras (or quantum groupoids in the
terminology of Nikshych and Vainerman \cite{NV}) were introduced
by B\"{o}hm, Nill and Szlach\'anyi \cite{bohm} as a new
generalization of Hopf algebras and groupoid algebras. A weak Hopf
algebra $H$ in a braided monoidal category \cite{IND} is an object that has
both, monoid and comonoid structure, with some relations between
them. The main difference with other Hopf algebraic constructions is that weak
Hopf algebras are coassociative but the coproduct is not required
to preserve the unit, equivalently, the counit is
not a monoid morphism. Some motivations to study weak Hopf
algebras come from the following facts: firstly, as group algebras
and their duals are the natural examples of Hopf algebras,
groupoid algebras and their duals provide examples of weak Hopf
algebras and, secondly, these algebraic structures have a
remarkable connection with the theory of algebra extensions,
important applications in the study of dynamical twists of Hopf
algebras and a deep link with quantum field theories and operator
algebras \cite{NV}, as well as they are useful tools in the
study of fusion categories in characteristic zero 
\cite{ENO}.  Moreover, Hayashi's face algebras
(see \cite{Ha}) are particular instances of weak Hopf algebras,
whose counital subalgebras are
commutative, and  Yamanouchi's
generalized Kac algebras \cite{Ya} are exactly
$C^{\ast}$-weak Hopf algebras with involutive antipode.

On the other hand, Hopf quasigroups are a generalization of Hopf algebras in the context of non associative algebra.
 Like in the quasi-Hopf setting, Hopf quasigroups are not
associative but the lack of this property is compensated by some
axioms involving the antipode. The concept of Hopf quasigroup is a
particular instance of the notion of unital coassociative
$H$-bialgebra introduced in \cite{PI2}.
It includes the example
of an enveloping algebra  of a Malcev algebra (see
\cite{Majidesfera} and \cite{PIS}) when the base ring has characteristic not equal to $2$ nor $3$, and in this sense Hopf quasigroups extend the notion of Hopf algebra in a parallel way that Malcev algebras extend the one of Lie algebra. On the other hand, it also contains as an example the notion of quasigroup algebra 
of an I.P. loop. Therefore, Hopf quasigroups unify I.P. loops and Malcev
algebras in the same way that Hopf algebras unify groups and Lie
algebras. Actually, Hopf quasigroups in a category of vector spaces were introduced
 by Klim and Majid in \cite{Majidesfera} in order to
understand the structure and relevant properties of the algebraic
$7$-sphere.

The main purposes of this paper are to introduce the notion of weak Hopf quasigroup as a new Hopf algebra generalization that  encompass weak Hopf algebras and Hopf quasigroups and to prove that the more relevant properties of these algebraic structures can be obtained under a unified approach, that is, we show that  the fundamental assertions proved in \cite{bohm} and \cite{IND} about weak Hopf algebras and in \cite{Majidesfera} for Hopf quasigroups can be obtained in this new setting. Also, we construct  a  family of examples working with bigroupoids, i.e. bicategories where every $1$-cell is an equivalence and every $2$-cell is an isomorphism. The organization of the paper is the following. After this introduction, in Section 2 we introduce weak  Hopf quasigroups and we explain in detail how the first non-trivial examples of this algebraic structures can be obtained considering bigroupoids. In Section 3 we discuss the consequences of the definition of weak  Hopf quasigroups obtaining the first relevant properties of this objects.  Finally, Section 4 is devoted to  prove the fundamental theorem of Hopf modules associated to a weak  Hopf quasigroups.

\section{Definitions and examples}

Throughout this paper $\mathcal C$ denotes a
strict  monoidal category with tensor product $\ot$
and unit object $K$. For each object $M$ in  $ {\mathcal
C}$, we denote the identity morphism by $id_{M}:M\rightarrow M$ and, for
simplicity of notation, given objects $M$, $N$, $P$ in ${\mathcal
C}$ and a morphism $f:M\rightarrow N$, we write $P\ot f$ for
$id_{P}\ot f$ and $f \ot P$ for $f\ot id_{P}$.

From now on we assume that ${\mathcal C}$  admits
split idempotents, i.e. for every morphism
$\nabla_{Y}:Y\rightarrow Y$ such that $\nabla_{Y}=\nabla_{Y}\co
\nabla_{Y}$ there exist an object $Z$ and morphisms
$i_{Y}:Z\rightarrow Y$ and $p_{Y}:Y\rightarrow Z$ such that
$\nabla_{Y}=i_{Y}\co p_{Y}$ and $p_{Y}\co i_{Y} =id_{Z}$. There is
no loss of generality in assuming that ${\mathcal
C}$ admits split idempotents, taking into account that, for a given category
${\mathcal C}$, there exists an universal embedding ${\mathcal
C}\rightarrow \hat{\mathcal C}$ such that $\hat{\mathcal C}$ admits
split idempotents, as was proved in  \cite{Karoubi}.

Also we assume that ${\mathcal C}$ is braided, that is:  for all $M$ and $N$ objects in ${\mathcal C}$,
there is a natural isomorphism $c_{M, N}:M\ot N\rightarrow N\ot M$,
called the braiding, satisfying the Hexagon Axiom (see \cite{JS}
for generalities). If the braiding satisfies $c_{N,M}\co
c_{M,N}=id_{M\ot N}$, the category ${\mathcal C}$ will be called
symmetric.

\begin{definition}
By a unital  magma in ${\mathcal C}$ we understand a triple
$A=(A, \eta_{A}, \mu_{A})$ where $A$ is an object in ${\mathcal C}$
and $\eta_{A}:K\rightarrow A$ (unit), $\mu_{A}:A\otimes A
\rightarrow A$ (product) are morphisms in ${\mathcal C}$ such that
$\mu_{A}\circ (A\otimes \eta_{A})=id_{A}=\mu_{A}\circ
(\eta_{A}\otimes A)$. If $\mu_{A}$ is associative, that is,
$\mu_{A}\circ (A\otimes \mu_{A})=\mu_{A}\circ (\mu_{A}\otimes A)$,
the unital magma will be called a monoid in ${\mathcal C}$.   Given two unital magmas
(monoids) $A= (A, \eta_{A}, \mu_{A})$ and $B=(B, \eta_{B},
\mu_{B})$, $f:A\rightarrow B$ is a morphism of unital magmas (monoids) 
if $\mu_{B}\circ (f\otimes f)=f\circ \mu_{A}$ and $
f\circ \eta_{A}= \eta_{B}$. 

By duality, a counital comagma in
${\mathcal C}$ is a triple ${D} = (D, \varepsilon_{D},
\delta_{D})$ where $D$ is an object in ${\mathcal C}$ and
$\varepsilon_{D}: D\rightarrow K$ (counit), $\delta_{D}:D\rightarrow
D\otimes D$ (coproduct) are morphisms in ${\mathcal C}$ such that
$(\varepsilon_{D}\otimes D)\circ \delta_{D}= id_{D}=(D\otimes
\varepsilon_{D})\circ \delta_{D}$. If $\delta_{D}$ is coassociative,
that is, $(\delta_{D}\otimes D)\circ \delta_{D}=
 (D\otimes \delta_{D})\circ \delta_{D}$, the counital comagma will be called a comonoid.
 If ${D} = (D, \varepsilon_{D},
 \delta_{D})$ and
${ E} = (E, \varepsilon_{E}, \delta_{E})$ are counital comagmas
(comonoids), $f:D\rightarrow E$ is  a morphism of counital comagmas
(comonoids) if $(f\otimes f)\circ \delta_{D}
=\delta_{E}\circ f$ and  $\varepsilon_{E}\circ f =\varepsilon_{D}.$

If  $A$, $B$ are unital magmas (monoids) in
${\mathcal C}$, the object $A\otimes B$ is a unital  magma (monoid)
in ${\mathcal C}$ where $\eta_{A\otimes B}=\eta_{A}\otimes \eta_{B}$
and $\mu_{A\otimes B}=(\mu_{A}\otimes \mu_{B})\circ (A\otimes
c_{B,A}\otimes B).$  In a dual way, if $D$, $E$ are counital
comagmas (comonoids) in ${\mathcal C}$, $D\otimes E$ is a  counital
comagma (comonoid) in ${\mathcal C}$ where $\varepsilon_{D\otimes
E}=\varepsilon_{D}\otimes \varepsilon_{E}$ and $\delta_{D\otimes
E}=(D\otimes c_{D,E}\otimes E)\circ( \delta_{D}\otimes \delta_{E}).$

Finally, if $D$ is a comagma and $A$ a magma, for two morphisms $f,g:D\rightarrow A$ with $f\ast g$ we will denote its convolution product in ${\mathcal C}$, that is 
$$f\ast g=\mu_{A}\co (f\ot g)\co \delta_{D}.$$
\end{definition}

\begin{definition}
\label{Weak-Hopf-quasigroup}   A weak Hopf quasigroup $H$   in ${\mathcal
C}$ is a unital magma $(H, \eta_H, \mu_H)$ and a comonoid $(H,\varepsilon_H, \delta_H)$ such that the following axioms hold:
\begin{itemize}
\item[(a1)] $\delta_{H}\co \mu_{H}=(\mu_{H}\ot \mu_{H})\co \delta_{H\ot H}.$
\item[(a2)] $\varepsilon_{H}\co \mu_{H}\co (\mu_{H}\ot
H)=\varepsilon_{H}\co \mu_{H}\co (H\ot \mu_{H})$
\item[ ]$= ((\varepsilon_{H}\co \mu_{H})\ot (\varepsilon_{H}\co
\mu_{H}))\co (H\ot \delta_{H}\ot H)$ 
\item[ ]$=((\varepsilon_{H}\co \mu_{H})\ot (\varepsilon_{H}\co
\mu_{H}))\co (H\ot (c_{H,H}^{-1}\co\delta_{H})\ot H).$
\item[(a3)]$(\delta_{H}\ot H)\co \delta_{H}\co \eta_{H}=(H\ot
\mu_{H}\ot H)\co ((\delta_{H}\co \eta_{H}) \ot (\delta_{H}\co
\eta_{H}))$  \item[ ]$=(H\ot (\mu_{H}\co c_{H,H}^{-1})\ot
H)\co ((\delta_{H}\co \eta_{H}) \ot (\delta_{H}\co \eta_{H})).$
\item[(a4)] There exists  $\lambda_{H}:H\rightarrow H$
in ${\mathcal C}$ (called the antipode of $H$) such that, if we denote the morphisms $id_{H}\ast \lambda_{H}$ by 
$\Pi_{H}^{L}$ (target morphism) and $\lambda_{H}\ast
id_{H}$ by $\Pi_{H}^{R}$ (source morphism):
\begin{itemize}
\item[(a4-1)] $\Pi_{H}^{L}=((\varepsilon_{H}\co
\mu_{H})\ot H)\co (H\ot c_{H,H})\co ((\delta_{H}\co \eta_{H})\ot
H).$
\item[(a4-2)] $\Pi_{H}^{R}=(H\ot(\varepsilon_{H}\co \mu_{H}))\co (c_{H,H}\ot H)\co
(H\ot (\delta_{H}\co \eta_{H})).$
\item[(a4-3)]$\lambda_{H}\ast \Pi_{H}^{L}=\Pi_{H}^{R}\ast \lambda_{H}= \lambda_{H}.$
\item[(a4-4)] $\mu_H\circ (\lambda_H\ot \mu_H)\circ (\delta_H\ot H)=\mu_{H}\co (\Pi_{H}^{R}\ot H).$
\item[(a4-5)] $\mu_H\circ (H\ot \mu_H)\circ (H\ot \lambda_H\ot
H)\circ (\delta_H\ot H)=\mu_{H}\co (\Pi_{H}^{L}\ot H).$
\item[(a4-6)] $\mu_H\circ(\mu_H\ot \lambda_H)\circ (H\ot
\delta_H)=\mu_{H}\co (H\ot \Pi_{H}^{L}).$
\item[(a4-7)] $\mu_H\circ (\mu_H\ot H)\circ (H\ot \lambda_H\ot H)\circ (H\ot \delta_H)=\mu_{H}\co (H\ot \Pi_{H}^{R}).$

\end{itemize}
\end{itemize}

Note that, if in the previous definition the triple $(H, \eta_H, \mu_H)$ is a monoid, we obtain the notion of weak Hopf algebra in a braided category introduced in \cite{AFG1} (see also \cite{IND}). Under this assumption, if ${\mathcal C}$ is symmetric, we have the monoidal version of the original definition of weak Hopf algebra introduced by B\"{o}hm, Nill and Szlach\'anyi in \cite{bohm}. On the other hand, if  $\varepsilon_H$ and $\delta_H$ are  morphisms of unital magmas, $\Pi_{H}^{L}=\Pi_{H}^{R}=\eta_{H}\ot \varepsilon_{H}$ and, as a consequence, we have  the notion of Hopf quasigroup defined  by Klim and Majid in \cite{Majidesfera} ( note that in this case there is not difference between the definitions for the symmetric and the braided  settings).
\end{definition}

\begin{example} 
\label{main-example}
In this example we will show that it is possible to obtain non-trivial examples of weak Hopf quasigroups working with bicategories in the sense of B\'enabou \cite{BEN}. A bicategory ${\mathcal B}$ consists of :
\begin{itemize}
\item[(b1)] A set ${\mathcal B}_{0}$, whose elements $x$ are called $0$-cells.
\item[(b2)] For each $x$, $y\in {\mathcal B}_{0}$, a category ${\mathcal B}(x,y)$ whose objects $f:x\rightarrow y$ are called $1$-cells and whose morphisms $\alpha:f \Rightarrow g$ are called $2$-cells. The composition of $2$-cells is called the vertical composition of $2$-cells and if $f$ is a $1$-cell in ${\mathcal B}(x,y)$, $x$ is called the source of $f$, represented by $s(f)$, and $y$ is called the target of $f$, denoted by $t(f)$. 
\item[(b3)] For each $x\in {\mathcal B}_{0}$, an object $1_{x}\in {\mathcal B}(x,x)$, called the identity of $x$; and for each $x,y,z\in {\mathcal B}_{0}$, a functor 
$${\mathcal B}(y,z)\times {\mathcal B}(x,y)\rightarrow {\mathcal B}(x,z)$$ 
which in objects is called the $1$-cell composition $(g,f)\mapsto g\co f$, and on arrows is called horizontal composition of $2$-cells: 
$$f,f^{\prime}\in {\mathcal B}(x,y), \;\; g,g^{\prime}\in {\mathcal B}(y,z), \; \alpha:f \Rightarrow f^{\prime}, \; \beta:g \Rightarrow g^{\prime}$$
$$(\beta, \alpha)\mapsto \beta\bullet \alpha:g\co f \Rightarrow g^{\prime}\co f^{\prime}$$ 
\item[(b4)] For each $f\in {\mathcal B}(x,y)$, $g\in {\mathcal B}(y,z)$, $h\in {\mathcal B}(z,w)$, an associative isomorphisms $\xi_{h,g,f}: (h\co g)\co f\Rightarrow h\co (g\co f)$; and for each $1$-cell $f$,  unit  isomorphisms $l_{f}:1_{t(f)}\co f\Rightarrow f$, $r_{f}:f\co 1_{s(f)}\Rightarrow f$, satisfying the following coherence axioms:
\begin{itemize}
\item[(b4-1)] The morphism $\xi_{h,g,f}$ is natural in $h$, $f$ and $g$ and $l_{f}$, $r_{f}$ are natural in $f$.
\item[(b4-2)] Pentagon axiom: $ \xi_{k,h,g\co f}\co \xi_{k\co h,g, f}=(id_{k}\bullet \xi_{ h,g, f})\co 
\xi_{k, h\co g, f}\co (\xi_{k,h,g}\bullet id_{f}).$ 
\item[(b4-3)] Triangle axiom: $r_{g}\bullet id_{f}=(id_{g}\bullet l_{f})\co \xi_{g,1_{t(f)},f}.$ 

\end{itemize}
\end{itemize}
A bicategory is normal if the unit  isomorphisms are identities. Every bicategory is biequivalent to a normal one. 
A $1$-cell $f$ is called an equivalence if there exists a $1$-cell $g:t(f)\rightarrow s(f)$ and two isomorphisms $g\co f\Rightarrow 1_{s(f)}$, $f\co g\Rightarrow 1_{t(f)}$.  In this case we will say that $g\in Inv(f)$ and, equivalently, $f\in Inv(g)$.

A bigroupoid is a bicategory where every $1$-cell is an equivalence and every $2$-cell is an isomorphism. We will say that a bigroupoid ${\mathcal B}$ is finite if ${\mathcal B}_{0}$ is finite and ${\mathcal B}(x,y)$ is small for all $x,y$. Note that if ${\mathcal B}$ is a  bigroupoid where ${\mathcal B}(x,y)$ is small for all $x,y$ and we pick a finite number of $0$-cells, considering the full sub-bicategory  generated by these $0$-cells, we have an example of finite bigroupoid.

Let ${\mathcal B}$ be a finite normal bigroupoid and denote  by ${\mathcal B}_{1}$ the set of $1$-cells. Let ${\Bbb F}$ be a field and ${\Bbb F}{\mathcal B}$ the direct product 
$${\Bbb F}{\mathcal B}=\bigoplus_{f\in {\mathcal B}_{1}}Ff.$$
The vector space ${\Bbb F}{\mathcal B}$ is a unital nonassociative algebra  where the product of two $1$-cells is equal to their $1$-cell composition if the latter is defined and $0$ otherwise, i.e.
$g.f=g\circ f$ if $s(g)=t(f)$ and
$g.f=0$ if $s(g)\neq t(f)$. The unit element is $$1_{{\Bbb F}{\mathcal B}}=\sum_{x\in {\mathcal B}_{0}}1_{x}.$$

Let $H={\Bbb F}{\mathcal B}/I({\mathcal B})$ be the quotient algebra where  $I({\mathcal B})$ is the ideal of ${\Bbb F}{\mathcal B}$ generated  by 
$$ h-g\co (f\co h),\; p-(p\co f)\co g,$$ 
with  $f\in {\mathcal B}_{1},$ 
$g\in Inv(f)$,  and $h,p \in {\mathcal B}_{1}$ such that $t(h)=s(f)$, $t(f)=s(p)$. In what follows, for any $1$-cell $f$ we denote its class in  $H$  by $[f]$. 

If there exists a $1$-cell $f$ in ${\mathcal B}(x,y)$ such that $[f]=0$ and we pick $g\in Inv(f)$ we have that $[1_{x}]=[g.f]=[g].[f]=0$ and  $[1_{y}]=[f.g]=[f].[g]=0$. Conversely, if $[1_{x}]=0$ or 
$[1_{y}]=0$ we have that $[f]=0$ because $f.1_{x}=1_{y}.f=f$. Therefore,  the following assertion holds: There exists a $1$-cell $f$ in ${\mathcal B}(x,y)$ such that $[f]=0$ if and only if $[h]=0$ for all  $1$-cell $h$ in ${\mathcal B}(x,y)$. Moreover, let $x,y,z, w$ be $0$-cells. If there exists a $1$-cell $f\in {\mathcal B}(x,y)$ satisfying that $[f]=0$ we have that $[1_{y}]=0$ and then $[h]=0$ for all $1$-cell $h$ in $ {\mathcal B}(y,z)$. As a consequence, $[1_{z}]=0$ and this clearly implies that 
 $[p]=0$ for all $1$-cell $p$  in ${\mathcal B}(z,w)$. Thus, if there exists a $1$-cell $f$ such that $[f]=0$ we obtain that $[h]=0$ for all $h\in {\mathcal B}_{1}$. According to this reasoning, there exists a $1$-cell $f$ such that $[f]=0$ if and only if $I({\mathcal B})=F{\mathcal B}$. Equivalently, $H$ is not null, if and only if $[f]\ne 0$ for all $f\in {\mathcal B}_{1}$.

Then, in the remainder of this section, we assume that $I({\mathcal B})$ is a proper ideal. Under this condition if $f\in {\mathcal B}_{1}$ and $g,h\in Inv(f)$ we have 
$$[g]=[g. (f. g)]=[g.1_{y}]=[1_{x}.g]=[(h.f).g]=[h].$$
Moreover, for all $f, f^{\prime}\in {\mathcal B}_{1}$ such that $[f]=[f^{\prime}]$, the following holds: if 
$s(f)\ne s(f^{\prime})$ we have 
$$[f]=[f.1_{s(f)}]=[f].[1_{s(f)}]=[f^{\prime}].[1_{s(f)}]=[f^{\prime}.1_{s(f)}]=0.$$
In a similar way, if $t(f)\ne t(f^{\prime})$ we obtain that $[f]=0.$ Thus,  $[f]=[f^{\prime}]$, clearly forces that $f$ and $f^{\prime}$ are $1$-cells in ${\mathcal B}(s(f), t(f))$. Moreover, if $f,f^{\prime}$ are $1$-cells in ${\mathcal B}(x, y)$ such that $[f]=[f^{\prime}]$ and $g\in Inv(f)$, $g^{\prime}\in Inv(f^{\prime})$
we have 
$$[g^{\prime}]=[1_{x}.g^{\prime}]=[(g.f).g^{\prime}]=([g].[f]).[g^{\prime}]=([g].[f^{\prime}]).[g^{\prime}]
=[(g.f^{\prime}).g^{\prime}]=[g].$$
Then, for a $1$-cell $f$ we denote by $[f]^{-1}$ the class of any $g\in Inv(f)$. Note that, in the previous equalities, we proved that $[f]^{-1}$ is independent of the choices of $g\in Inv(f)$ and 
$f^{\prime}$ such that $[f]=[f^{\prime}]$.

Therefore, the vector space $H$ with the product $\mu_{H}([g]\ot [f])=[g.f]$ and  the  unit $$\eta_{ H}(1_{{\Bbb F}})=[1_{{\Bbb F}{\mathcal B}}]=\sum_{x\in {\mathcal B}_{0}}[1_{x}]$$ is a unital magma. Also, it is easy to show that $H$ is a comonoid with coproduct $\delta_{H}([f])=[f]\ot [f]$ and counit $\varepsilon_{H}([f])=1_{{\Bbb F}}$. Moreover,  the morphism $\lambda_{H}:H\rightarrow H$, $\lambda_{H}([f])=[f]^{-1}$ is well-defined and $H=(H,\eta_{H}, \mu_{H}, \varepsilon_{H}, \delta_{H}, \lambda_{H})$ is a weak Hopf quasigroup. Indeed: First note that, for all $1$-cells $f,g$ we have 
$$(\delta_{H}\co \mu_{H}) ([g]\ot [f])=[g.f]\ot [g.f]$$
if $s(g)=t(f)$ and $0$ otherwise. On the other hand, 
$$((\mu_{H}\ot \mu_{H})\co \delta_{H\ot H})([g]\ot [f])=(\mu_{H} ([g]\ot [f])\ot \mu_{H} ([g]\ot [f]))=[g.f]\ot [g.f]$$ if $s(g)=t(f)$ and $0$ otherwise because $c_{H,H}([g]\ot [f])=[f]\ot [g]$. Therefore,  (a1) of Definition \ref{Weak-Hopf-quasigroup} holds.

If $f,g, h$ are $1$-cells we have the following equalities:
$$(\varepsilon_{H}\co \mu_{H}\co (\mu_{H}\ot
H))([h]\ot [g]\ot [f]) =1_{{\Bbb F}}=(\varepsilon_{H}\co \mu_{H}\co (H\ot \mu_{H}))([h]\ot [g]\ot [f])$$
when $s(h)=t(g),\; s(g)=t(f)$ and 
$$(\varepsilon_{H}\co \mu_{H}\co (\mu_{H}\ot
H))([h]\ot [g]\ot [f]) =0=(\varepsilon_{H}\co \mu_{H}\co (H\ot \mu_{H}))([h]\ot [g]\ot [f])$$
otherwise. Also, 
$$(((\varepsilon_{H}\co \mu_{H})\ot (\varepsilon_{H}\co
\mu_{H}))\co (H\ot \delta_{H}\ot H))([h]\ot [g]\ot [f]) $$
$$= ((\varepsilon_{H}\co \mu_{H})([h]\ot [g]))\ot ((\varepsilon_{H}\co
\mu_{H}))([g]\ot [f]))=1_{{\Bbb F}}$$
if $s(h)=t(g),\; s(g)=t(f)$ and 
$$(((\varepsilon_{H}\co \mu_{H})\ot (\varepsilon_{H}\co
\mu_{H}))\co (H\ot \delta_{H}\ot H))([h]\ot [g]\ot [f])=0$$
 otherwise. Then (a2) of Definition \ref{Weak-Hopf-quasigroup} holds because in this case $c_{H,H}=c_{H,H}^{-1}$ and $\delta_{H}\co c_{H,H}=\delta_{H}$ (i.e. $H$ is cocommutative).

To prove (a3) first note that 
$$((\delta_{H}\ot H)\co \delta_{H}\co \eta_{H}) (1_{{\Bbb F}})=\sum_{x\in {\mathcal B}_{0}}[1_{x}]\ot [1_{x}]\ot [1_{x}]$$
Then (a3) holds because:
$$((H\ot
\mu_{H}\ot H)\co ((\delta_{H}\co \eta_{H}) \ot (\delta_{H}\co
\eta_{H})))(1_{{\Bbb F}}\ot 1_{{\Bbb F}})$$
$$=(H\ot
\mu_{H}\ot H)(\sum_{x\in {\mathcal B}_{0}}[1_{x}]\ot [1_{x}] \ot \sum_{y\in {\mathcal B}_{0}}[1_{y}]\ot [1_{y}]) $$
$$=\sum_{x,y\in {\mathcal B}_{0}}[1_{x}]\ot [1_{x}.1_{y}] \ot [1_{y}]=\sum_{x\in {\mathcal B}_{0}}[1_{x}]\ot [1_{x}] \ot [1_{x}].$$

To prove the antipode identities first note that
\begin{equation}
\label{pi-ex}
\Pi_{H}^{L}([f])=[1_{t(f)}], \;\; \Pi_{H}^{R}([f])=[1_{t(s)}]
\end{equation}
for all $1$-cell $f$. 

Then, (a4-1) and (a4-2) hold because, for all $1$-cell $f$,  
$$(((\varepsilon_{H}\co
\mu_{H})\ot H)\co (H\ot c_{H,H})\co ((\delta_{H}\co \eta_{H})\ot
H))([f])$$
$$=(((\varepsilon_{H}\co
\mu_{H})\ot H)\co (H\ot c_{H,H}))(\sum_{x\in {\mathcal B}_{0}}[1_{x}]\ot [1_{x}] \ot [f])$$
$$=\sum_{x\in {\mathcal B}_{0}} \varepsilon_{H}([1_{x}.f])\ot [1_{x}]=[1_{t(f)}]$$
and, by a similar calculus, 
$$((H\ot(\varepsilon_{H}\co \mu_{H}))\co (c_{H,H}\ot H)\co
(H\ot (\delta_{H}\co \eta_{H})))([f])=[1_{s(f)}].$$

Also, if $f\in {\mathcal B}_{1}$, by (\ref{pi-ex}), 
$$(\lambda_{H}\ast \Pi_{H}^{L})([f])=[f]^{-1}.[1_{t(f)}]=[f]^{-1}=\lambda_{H}([f]),$$ 
$$(\Pi_{H}^{R}\ast \lambda_{H})([f])=[1_{s(f)}].[f]^{-1}=[f]^{-1}=\lambda_{H}([f])$$
and then (a4-3) holds.

The proof for (a4-4) is the following: It follows easily that for two $1$-cells $f, h$ we have that 
$$(\mu_{H}\co (\Pi_{H}^{R}\ot H))([h]\ot [f])=[f]$$
if $s(h)=t(f)$ and $0$ otherwise.  On the other hand, 
$$(\mu_H\circ (\lambda_H\ot \mu_H)\circ (\delta_H\ot H))([h]\ot [f])=\mu_{H}([h]^{-1}\ot [h.f])$$
if $s(h)=t(f)$ and $0$ otherwise. Therefore, if $m\in Inv(h)$ and $s(h)=t(f)$ the equality
$$\mu_{H}([h]^{-1}\ot [h.f])=[m.(h.f)]=[f]$$
holds and thus (a4-4)  holds.

If $f$, $h$ are $1$-cells we have 
$$(\mu_{H}\co (\Pi_{H}^{L}\ot H))([h]\ot [f])=[f]$$
if $t(h)=t(f)$ and $0$ otherwise.  Moreover,  let $m\in Inv(h)$, then 
$$(\mu_H\circ (H\ot \mu_H)\circ (H\ot \lambda_H\ot
H)\circ (\delta_H\ot H))([h]\ot [f])=\mu_{H}([h]\ot [m.f])$$ 
$$=[h.(m.f)]=[f]$$
if $t(h)=t(f)$  and $0$ otherwise. Therefore, (a4-5)  holds. 

The proofs for (a4-6) and (a4-7) are similar and  the details are left to the reader.

Note that, in this example, if ${\mathcal B}_{0}=\{x\}$ we obtain that $H$ is a Hopf quasigroup. Moreover, if $\vert {\mathcal B}_{0}\vert >1$ and the product defined in $H$ is associative we have an example of weak Hopf algebra. 

\end{example}

\section{Basic properties for weak Hopf quasigroups}

In this section we will show the main properties of weak Hopf quasigroups.  First, note that by the naturality of the braiding, for the morphisms target and source the following equalities hold:
\begin{equation}
\label{new-pi-1}
\Pi_{H}^{L}=((\varepsilon_{H}\co \mu_{H}\co c^{-1}_{H,H})\ot
H)\co (H\ot (\delta_{H}\co \eta_{H}))
\end{equation}
$$=(H\ot (\varepsilon_{H}\co
\mu_{H}))\co (( c^{-1}_{H,H}\co \delta_{H}\co \eta_{H})\ot H),$$
\begin{equation}
\label{new-pi-2}
\Pi_{H}^{R}=(H\ot (\varepsilon_{H}\co \mu_{H}\co
c^{-1}_{H,H}))\co ((\delta_{H}\co \eta_{H})\ot H)=
\end{equation}
 $$
((\varepsilon_{H}\co \mu_{H})\ot H)\co (H\ot (c^{-1}_{H,H} \co
\delta_{H}\co \eta_{H})).
$$

\begin{proposition}
Let $H$ be a weak Hopf quasigroup. The following equalities hold:
\begin{equation}
\label{pi-l}
\Pi_{H}^{L}\ast id_{H}=id_{H}\ast \Pi_{H}^{R}=id_{H},
\end{equation}
\begin{equation}
\label{pi-eta}
\Pi_{H}^{L}\co\eta_{H}=\eta_{H}=\Pi_{H}^{R}\co\eta_{H},
\end{equation}
\begin{equation}
\label{pi-varep}
\varepsilon_{H}\co \Pi_{H}^{L}=\varepsilon_{H}=\varepsilon_{H}\co \Pi_{H}^{R}.
\end{equation}

\end{proposition}

\begin{proof}
By the definition of $\Pi_{H}^{L}$  and (a1) of Definition \ref{Weak-Hopf-quasigroup} we have 
$$\Pi_{H}^{L}\ast id_{H}=(\varepsilon_{H}\ot H)\co \mu_{H\ot H}\co ((\delta_{H}\co \eta_{H})\ot \delta_{H})=
(\varepsilon_{H}\ot H)\co \delta_{H}\co \mu_{H}\co (\eta_{H}\ot H)=id_{H}.$$
We can now proceed analogously to the proof of $id_{H}\ast \Pi_{H}^{R}=id_{H}$. Finally (\ref{pi-eta}) and (\ref{pi-varep}) follow easily from the definitions of $\Pi_{H}^{L}$ and $\Pi_{H}^{R}$.

\end{proof}

\begin{proposition}
\label{antipode-1} The antipode of a  weak Hopf quasigroup $H$ is unique and leaves the unit and the counit invariant, i.e. $\lambda_{H}\co \eta_{H}=\eta_{H}$ and $\varepsilon_{H}\co\lambda_{H}=\varepsilon_{H}$.
\end{proposition}

\begin{proof} Let $\lambda_{H}, s_{H}:H\rightarrow H$ two morphisms satisfying (a4) of Definition \ref{Weak-Hopf-quasigroup}. Then, 
$$s_{H}= (s_{H}\ast H)\ast s_{H} = (\lambda_{H}\ast H)\ast s_{H} = \mu_{H}\co (\mu_{H}\ot s_{H})\co (\lambda_{H}\ot \delta_{H})\co \delta_{H}= \lambda_{H}\ast \Pi_{H}^{L}=\lambda_{H}, $$
where the first and the last equalities  follow by (a4-3) of Definition \ref{Weak-Hopf-quasigroup}, the second one by 
 (a4-2) of Definition \ref{Weak-Hopf-quasigroup}, the third one by the coassociativity of $\delta_{H}$ and the fourth one by 
(a4-6) of Definition \ref{Weak-Hopf-quasigroup}. 

On the other hand,  by (a4-3), (a3) of Definition \ref{Weak-Hopf-quasigroup}, the naturality of the braiding and (\ref{pi-eta}), we have 
\begin{itemize}
\item[ ]$\hspace{0.38cm} \lambda_{H}\co \eta_{H}$

\item[ ]$= (\Pi_{H}^{R}\ast \lambda_{H})\co \eta_{H} $

\item[ ]$= \mu_{H}\co (H\ot (\varepsilon_{H}\co \mu_{H}\co c_{H,H}^{-1})\ot
\lambda_{H})\co ((\delta_{H}\co \eta_{H}) \ot (\delta_{H}\co \eta_{H})) $

\item[ ]$=\mu_{H}\co (((H\ot \varepsilon_{H})\co \delta_{H})\ot \lambda_{H})\co \delta_{H}\co \eta_{H} $

\item[ ]$=\Pi_{H}^{L}\co \eta_{H} $

\item[ ]$=\eta_{H}. $
\end{itemize}

The proof for the equalities involving the counit follows a similar pattern but 
 using (a2) of Definition \ref{Weak-Hopf-quasigroup} and (\ref{pi-varep}) instead of (a3) and (\ref{pi-eta}) respectively.
\end{proof}

\begin{definition}
\label{pi-bar}
Let $H$ be a weak Hopf quasigroup. We define the morphisms $\overline{\Pi}_{H}^{L}$ and $\overline{\Pi}_{H}^{R}$ by 
$$\overline{\Pi}_{H}^{L}=(H\ot (\varepsilon_{H}\co \mu_{H}))\co ((\delta_{H}\co \eta_{H})\ot H),$$
and 
$$\overline{\Pi}_{H}^{R}=((\varepsilon_{H}\co \mu_{H})\ot H)\co (H\ot (\delta_{H}\co \eta_{H})).$$

\end{definition}

\begin{proposition}
\label{idemp}
Let $H$ be a weak Hopf quasigroup. The morphisms $\Pi_{H}^{L}$, $\Pi_{H}^{R}$, $\overline{\Pi}_{H}^{L}$ and 
$\overline{\Pi}_{H}^{R}$ are idempotents.
\end{proposition}

\begin{proof} First, by (\ref{new-pi-1}) and (a3) of Definition \ref{Weak-Hopf-quasigroup} we have that 

\begin{itemize}
\item[ ]$\hspace{0.38cm}\Pi_{H}^{L}\co \Pi_{H}^{L} $

\item[ ]$=((\varepsilon_{H}\co \mu_{H}\co c_{H,H}^{-1})\ot
(\varepsilon_{H}\co \mu_{H}\co c_{H,H}^{-1})\ot H )\co (H\ot
(\delta_{H}\co \eta_{H})\ot (\delta_{H}\co \eta_{H}))  $

\item[ ]$=((\varepsilon_{H}\co \mu_{H}\co c_{H,H}^{-1})\ot
\varepsilon_{H}\ot H )\co (H\ot ((\delta_{H}\ot H)\co\delta_{H}\co
\eta_{H})) $

\item[ ]$=\Pi_{H}^{L}.
 $
\end{itemize}

With the same reasoning but using (\ref{new-pi-2}) instead of (\ref{new-pi-1}) we  prove that $\Pi_{H}^{R}$ is an idempotent morphism. Finally, by (a3) of Definition \ref{Weak-Hopf-quasigroup}, $\overline{\Pi}_{H}^{L}\co
\overline{\Pi}_{H}^{L} =\overline{\Pi}_{H}^{L}$ and $\overline{\Pi}_{H}^{R}\co
\overline{\Pi}_{H}^{R} =\overline{\Pi}_{H}^{R}.$
\end{proof}

\begin{proposition}
\label{mu-idemp}
Let $H$ be a weak Hopf quasigroup. The following identities hold:
\begin{equation}
\label{mu-pi-l}
\mu_{H}\co (H\ot \Pi_{H}^{L})=((\varepsilon_{H}\co
\mu_{H})\ot H)\co (H\ot c_{H,H})\co (\delta_{H}\ot
H),
\end{equation}
\begin{equation}
\label{mu-pi-r}
\mu_{H}\co (\Pi_{H}^{R}\ot H)=(H\ot(\varepsilon_{H}\co \mu_{H}))\co (c_{H,H}\ot H)\co
(H\ot \delta_{H}),
\end{equation}
\begin{equation}
\label{mu-pi-l-var}
\mu_{H}\co (H\ot \overline{\Pi}_{H}^{L})=(H\ot (\varepsilon_{H}\co
\mu_{H}))\co (\delta_{H}\ot H),
\end{equation}
\begin{equation}
\label{mu-pi-r-var}
\mu_{H}\co (\overline{\Pi}_{H}^{R}\ot H)=((\varepsilon_{H}\co
\mu_{H})\ot H)\co (H\ot \delta_{H}).
\end{equation}

\end{proposition}

\begin{proof} 
We first prove (\ref{mu-pi-l}). 
\begin{itemize}
\item[ ]$\hspace{0.38cm}\mu_{H}\co ( H\ot \Pi_{H}^{L}) $

\item[ ]$=(\varepsilon_{H}\ot H)\co \delta_{H}\co \mu_{H}\co (
H\ot \Pi_{H}^{L})  $

\item[ ]$=(\varepsilon_{H}\ot H)\co \mu_{H\ot H}\co (\delta_{H}\ot \delta_{H})\co ( H\ot (((\varepsilon_{H}\co \mu_{H}\co c^{-1}_{H,H})\ot
H)\co (H\ot (\delta_{H}\co \eta_{H}))))$

\item[ ]$= ((((\varepsilon_{H}\co \mu_{H})\ot (\varepsilon_{H}\co
\mu_{H}))\co (H\ot (c_{H,H}^{-1}\co\delta_{H})\ot H))\ot \mu_{H})\co (H\ot H\ot c_{H,H}\ot H) \co (H\ot c_{H,H}\ot H\ot H) $
\item[ ]$\hspace{0.38cm}\co (\delta_{H}\ot c_{H,H}^{-1}\ot H)\co (H\ot H\ot (\delta_{H}\co \eta_{H}))$

\item[ ]$= ((\varepsilon_{H}\co \mu_{H}\co (\mu_{H}\ot H))\ot \mu_{H})\co (H\ot H\ot c_{H,H}\ot H)\co (H\ot c_{H,H}\ot c_{H,H})\co (\delta_{H}\ot (\delta_{H}\co \eta_{H})\ot H) $

\item[ ]$=((\varepsilon_{H}\co \mu_{H})\ot H)\co (H\ot c_{H,H})\co ((\mu_{H\ot H}\co (\delta_{H}\ot \delta_{H}))\ot H)\co 
(H\ot \eta_{H}\ot H)$

\item[ ]$=((\varepsilon_{H}\co\mu_{H})\ot H)\co (H\ot c_{H,H})\co (\delta_{H}\ot
H).$
\end{itemize}

In the last identities, the first one follows by the properties of the counit, the second one follows by (\ref{new-pi-1}) and (a1) of Definition  \ref{Weak-Hopf-quasigroup}. The third and the fifth ones rely on the naturality of $c$. The fourth equality is a consequence of (a2) of Definition  \ref{Weak-Hopf-quasigroup} and  finally the last one follows by (a1) of Definition  \ref{Weak-Hopf-quasigroup} and the properties of the unit. 

The proof for (\ref{mu-pi-r}) is similar but in the second step we must use (\ref{new-pi-2}) instead of (\ref{new-pi-1}).  To finish the proof we show that (\ref{mu-pi-l-var}) holds. The proof for (\ref{mu-pi-r-var}) is similar.

\begin{itemize}
\item[ ]$\hspace{0.38cm}\mu_{H}\co ( H\ot \overline{\Pi}_{H}^{L}) $

\item[ ]$=(\varepsilon_{H}\ot H)\co \delta_{H}\co \mu_{H}\co (
H\ot \overline{\Pi}_{H}^{L})  $

\item[ ]$=(\mu_{H}\ot (\varepsilon_{H}\co \mu_{H})\ot (\varepsilon_{H}\co \mu_{H}))\co (\delta_{H\ot H}\ot H\ot H)\co 
(H\ot (\delta_{H}\co \eta_{H})\ot H) $

\item[ ]$=  (\mu_{H}\ot (((\varepsilon_{H}\co \mu_{H})\ot (\varepsilon_{H}\co \mu_{H}))\co (H\ot \delta_{H}\ot H)))\co 
(\delta_{H\ot H}\ot H)\co (H\ot \eta_{H}\ot H)$

\item[ ]$= (\mu_{H}\ot (\varepsilon_{H}\co \mu_{H}\co (\mu_{H}\ot H)))\co (\delta_{H\ot H}\ot H)\co (H\ot \eta_{H}\ot H)$

\item[ ]$= (H\ot (\varepsilon_{H}\co \mu_{H}))\co (((\delta_{H}\co \mu_{H})\co (H\ot \eta_{H}))\ot H)$

\item[ ]$=(H\ot (\varepsilon_{H}\co
\mu_{H}))\co (\delta_{H}\ot H) .$
\end{itemize}
The first  equality follows by the counit properties, the second and the fifth ones by (a1) of Definition \ref{Weak-Hopf-quasigroup}, the third one follows from the coassociativity of $\delta_{H}$, the fourth one  by (a2)  of Definition \ref{Weak-Hopf-quasigroup}, and the sixth one by the properties of the unit.

\end{proof}

\begin{remark}
\label{Remark-mu-idemp}
Note that if we compose with $\varepsilon_{H}$ in the equalities (\ref{mu-pi-l}), (\ref{mu-pi-r}), (\ref{mu-pi-l-var}) and 
(\ref{mu-pi-r-var}) we obtain 
\begin{equation}
\label{mu-pi-l-varep}
\varepsilon_{H}\co \mu_{H}\co (H\ot \Pi_{H}^{L})=\varepsilon_{H}\co
\mu_{H},
\end{equation}
\begin{equation}
\label{mu-pi-r-varep}
\varepsilon_{H}\co\mu_{H}\co (\Pi_{H}^{R}\ot H)=\varepsilon_{H}\co \mu_{H},
\end{equation}
\begin{equation}
\label{mu-pi-l-var-varep}
\varepsilon_{H}\co\mu_{H}\co (H\ot \overline{\Pi}_{H}^{L})=\varepsilon_{H}\co
\mu_{H},
\end{equation}
\begin{equation}
\label{mu-pi-r-var-varep}
\varepsilon_{H}\co\mu_{H}\co (\overline{\Pi}_{H}^{R}\ot H)=\varepsilon_{H}\co
\mu_{H}.
\end{equation}

\end{remark}

\begin{proposition}
\label{delta-idemp}
Let $H$ be a weak Hopf quasigroup. The following identities hold:
\begin{equation}
\label{delta-pi-l}
 (H\ot \Pi_{H}^{L})\co \delta_{H}=(\mu_{H}\ot H)\co (H\ot c_{H,H})\co ((\delta_{H}\co \eta_{H})\ot
H),
\end{equation}
\begin{equation}
\label{delta-pi-r}
(\Pi_{H}^{R}\ot H)\co \delta_{H}=(H\ot \mu_{H})\co (c_{H,H}\ot H)\co
(H\ot (\delta_{H}\co \eta_{H})),
\end{equation}
\begin{equation}
\label{delta-pi-l-var}
(\overline{\Pi}_{H}^{L}\ot H)\co \delta_{H}=(H\ot 
\mu_{H})\co ((\delta_{H}\co \eta_{H})\ot H),
\end{equation}
\begin{equation}
\label{delta-pi-r-var}
 (H\ot \overline{\Pi}_{H}^{R})\co \delta_{H}=(
\mu_{H}\ot H)\co (H\ot (\delta_{H}\co \eta_{H})).
\end{equation}
\end{proposition}

\begin{proof}
The proof for (\ref{delta-pi-l}) is the following:
\begin{itemize}
\item[ ]$\hspace{0.38cm} (\mu_{H}\ot H)\co (H\ot c_{H,H})\co ((\delta_{H}\co \eta_{H})\ot
H)$

\item[ ]$=  (((H\ot \varepsilon_{H})\co \delta_{H}\co\mu_{H})\ot H)\co (H\ot c_{H,H})\co ((\delta_{H}\co \eta_{H})\ot
H) $

\item[ ]$= (\mu_{H}\ot (\varepsilon_{H}\co \mu_{H})\ot H)\co (H\ot c_{H,H}\ot c_{H,H})\co (\delta_{H}\ot c_{H,H}\ot H)\co ((\delta_{H}\co \eta_{H})\ot \delta_{H})  $

\item[ ]$= (H\ot ((( \varepsilon_{H}\co \mu_{H})\ot H)\co (H\ot c_{H,H})\co (\delta_{H}\ot H)))\co  (( (\mu_{H}\ot H)\co (H\ot c_{H,H})\co ((\delta_{H}\co \eta_{H})\ot
H))\ot H)\co \delta_{H}$

\item[ ]$=((\mu_{H}\co c_{H,H}^{-1})\ot ((( \varepsilon_{H}\co \mu_{H})\ot H) \co (H\ot c_{H,H})))\co (H\ot ((H\ot \delta_{H})\co \delta_{H}\co \eta_{H})\ot H) \co \delta_{H}$

\item[ ]$=((\mu_{H}\co c_{H,H}^{-1})\ot ( \varepsilon_{H}\co \mu_{H})\ot H)\co (H\ot  H \ot (\mu_{H}\co c_{H,H}^{-1})\ot c_{H,H})\co (H\ot (\delta_{H}\co \eta_{H})\ot  (\delta_{H}\co \eta_{H})\ot H)\co \delta_{H}  $

\item[ ]$=(\mu_{H}\ot (((\varepsilon_{H}\co \mu_{H}\co (\mu_{H}\ot H))\ot H)\co (H\ot H\ot c_{H,H})\co (H\ot c_{H,H}\ot H)\co 
((\delta_{H}\co \eta_{H})\ot H\ot H)))  $
\item[ ]$\hspace{0.38cm}\co \delta_{H\ot H}\co (\eta_{H}\ot H) $

\item[ ]$=(\mu_{H}\ot (((\varepsilon_{H}\co \mu_{H}\co (H\ot \mu_{H}))\ot H)\co (H\ot H\ot c_{H,H})\co (H\ot c_{H,H}\ot H)\co 
((\delta_{H}\co \eta_{H})\ot H\ot H)))  $
\item[ ]$\hspace{0.38cm}\co \delta_{H\ot H}\co (\eta_{H}\ot H) $
\item[ ]$= (H\ot \Pi_{H}^{L})\co (\mu_{H}\ot \mu_{H})\co \delta_{H\ot H}\co (\eta_{H}\ot H)  $

\item[ ]$= (H\ot \Pi_{H}^{L})\co \delta_{H}. $

\end{itemize}

In the last equalities the first one follows by the counit properties, the second one by (a1) of Definition \ref{Weak-Hopf-quasigroup} and the naturality of $c$. In the third one we used the naturality of $c$ and the coassociativity of $\delta_{H}$. The fourth and the sixth ones are consequence of the equality 
\begin{equation}
\label{aux-1}
(\mu_{H}\ot H)\co (H\ot c_{H,H})\co ((\delta_{H}\co \eta_{H})\ot
H)=((\mu_{H}\co c_{H,H}^{-1})\ot H)\co (H\ot (\delta_{H}\co \eta_{H}))
\end{equation}
and the fifth one follows by (a3) of Definition \ref{Weak-Hopf-quasigroup}. The seventh one relies on (a2) of Definition \ref{Weak-Hopf-quasigroup}, the eight one follows by the naturality of $c$ and the last one by the unit properties.

The proof for (\ref{delta-pi-r}) is similar but using 
\begin{equation}
\label{aux-2}
(H\ot \mu_{H})\co (c_{H,H}\ot H)\co
(H\ot (\delta_{H}\co \eta_{H}))=(H\ot (\mu_{H}\co c_{H,H}^{-1}))\co ((\delta_{H}\co \eta_{H})\ot H)
\end{equation}
instead of (\ref{aux-1}). 

Moreover, (\ref{delta-pi-l-var}) holds because

\begin{itemize}
\item[ ]$\hspace{0.38cm} (\overline{\Pi}_{H}^{L}\ot H)\co \delta_{H}$

\item[ ]$= (\overline{\Pi}_{H}^{L}\ot H)\co \delta_{H}\co \mu_{H}\co (\eta_{H}\ot H) $

\item[ ]$= (\overline{\Pi}_{H}^{L}\ot H)\co (\mu_{H}\ot \mu_{H})\co \delta_{H\ot H}\co (\eta_{H}\ot H)$

\item[ ]$= (H\ot (\varepsilon_{H}\co \mu_{H}\co (H\ot \mu_{H}))\ot \mu_{H})\co ((\delta_{H}\co \eta_{H})\ot (\delta_{H\ot H}\co (\eta_{H}\ot H)))$

\item[ ]$= (H\ot ( ((\varepsilon_{H}\co \mu_{H})\ot (\varepsilon_{H}\co
\mu_{H}))\co (H\ot \delta_{H}\ot H))\ot \mu_{H})\co ((\delta_{H}\co \eta_{H})\ot (\delta_{H\ot H}\co (\eta_{H}\ot H)))$

\item[ ]$=(H\ot (\varepsilon_{H}\co \mu_{H})\ot ((\varepsilon_{H}\ot H)\co \delta_{H}\co \mu_{H}))\co ((\delta_{H}\co \eta_{H})\ot (\delta_{H}\co \eta_{H})\ot H) $

\item[ ]$= (H\ot \varepsilon_{H}\ot \mu_{H})\co (((\delta_{H}\ot H)\co \delta_{H}\co \eta_{H})\ot H) $

\item[ ]$=(H\ot \mu_{H})\co ((\delta_{H}\co \eta_{H})\ot H).   $

\end{itemize}

The first equality follows by the unit properties, the second one by (a1) of Definition \ref{Weak-Hopf-quasigroup} and the third one by the definition of $\overline{\Pi}_{H}^{L}$. The fourth equality relies on (a2) of Definition \ref{Weak-Hopf-quasigroup}. The fifth one is a consequence of the coassociativity of $\delta_{H}$ and the sixth one follows by (a3) of Definition \ref{Weak-Hopf-quasigroup}. Finally the last one holds by the properties of the counit.

The proof of (\ref{delta-pi-r-var}) is similar to the developed for (\ref{delta-pi-l-var}) and we leave it to the reader.

\end{proof}

\begin{remark}
\label{Remark-delta-idemp}
Note that if we compose with $\eta_{H}$ in the equalities (\ref{delta-pi-l}), (\ref{delta-pi-r}), (\ref{delta-pi-l-var}) and 
(\ref{delta-pi-r-var}) we obtain 
\begin{equation}
\label{delta-pi-l-eta}
 (H\ot \Pi_{H}^{L})\co \delta_{H}\co \eta_{H}=\delta_{H}\co \eta_{H},
\end{equation}
\begin{equation}
\label{delta-pi-r-eta}
(\Pi_{H}^{R}\ot H)\co \delta_{H}\co \eta_{H}=\delta_{H}\co \eta_{H},
\end{equation}
\begin{equation}
\label{delta-pi-l-var-eta}
(\overline{\Pi}_{H}^{L}\ot H)\co \delta_{H}\co \eta_{H}=\delta_{H}\co \eta_{H},
\end{equation}
\begin{equation}
\label{delta-pi-r-var-eta}
 (H\ot \overline{\Pi}_{H}^{R})\co \delta_{H}\co \eta_{H}=\delta_{H}\co \eta_{H}.
\end{equation}

\end{remark}

As a consequence of Propositions \ref{mu-idemp} and \ref{delta-idemp} we can get other useful identities.

\begin{proposition}
\label{pi-delta-mu-pi}
Let $H$ be a wek Hopf quasigroup. The following identities hold:
\begin{equation}
\label{pi-delta-mu-pi-1}
\Pi^{L}_{H}\circ \mu_{H}\circ (H\ot
\Pi^{L}_{H})=\Pi^{L}_{H}\circ \mu_{H},
\end{equation}
\begin{equation}
\label{pi-delta-mu-pi-2}
\Pi^{R}_{H}\circ
\mu_{H}\circ (\Pi^{R}_{H}\ot H)=\Pi^{R}_{H}\circ \mu_{H},
\end{equation}
\begin{equation}
\label{pi-delta-mu-pi-3}
(H\ot \Pi^{L}_{H})\circ \delta_{H}\circ
\Pi^{L}_{H}=\delta_{H}\circ \Pi^{L}_{H},
\end{equation}
\begin{equation}
\label{pi-delta-mu-pi-4}
( \Pi^{R}_{H}\ot
H)\circ \delta_{H}\circ \Pi^{R}_{H}=\delta_{H}\circ \Pi^{R}_{H}.
\end{equation}
\end{proposition}

\begin{proof}
The equality (\ref{pi-delta-mu-pi-1}) holds because:
\begin{itemize}
\item[ ]$\hspace{0.38cm}\Pi^{L}_{H}\circ \mu_{H}\circ (H\ot
\Pi^{L}_{H}) $

\item[ ]$=((\varepsilon_{H}\co
\mu_{H})\ot \Pi^{L}_{H})\co (H\ot c_{H,H})\co (\delta_{H}\ot
H)  $

\item[ ]$=((\varepsilon_{H}\co
\mu_{H})\ot H)\co (H\ot c_{H,H})\co (((H\ot \Pi^{L}_{H})\co\delta_{H})\ot
H) $

\item[ ]$=  ((\varepsilon_{H}\co \mu_{H}\co (\mu_{H}\ot H))\ot H)\co (H\ot H\ot \ot c_{H,H})\co (H\ot c_{H,H}\ot H)\co ((\delta_{H}\co \eta_{H})\ot H\ot H)$

\item[ ]$= ((\varepsilon_{H}\co \mu_{H}\co (H\ot \mu_{H}))\ot H)\co (H\ot H\ot \ot c_{H,H})\co (H\ot c_{H,H}\ot H)\co ((\delta_{H}\co \eta_{H})\ot H\ot H)$

\item[ ]$= \Pi^{L}_{H}\circ \mu_{H}. $

\end{itemize}

In the previous equalities, the first one follows by (\ref{mu-pi-l}), the second and the fifth ones by the naturality of $c$, the third one by (\ref{delta-pi-l}) and the fourth one by (a2) of Definition \ref{Weak-Hopf-quasigroup}. 

The proof for (\ref{pi-delta-mu-pi-1}) is similar. To finish we will show that (\ref{pi-delta-mu-pi-3}) holds (using the same reasoning we obtain (\ref{pi-delta-mu-pi-4})). Indeed: 

\begin{itemize}
\item[ ]$\hspace{0.38cm} (H\ot \Pi^{L}_{H})\circ \delta_{H}\circ
\Pi^{L}_{H}$

\item[ ]$=(\mu_{H}\ot H)\co (H\ot c_{H,H})\co ((\delta_{H}\co \eta_{H})\ot
\Pi^{L}_{H})  $

\item[ ]$=((\mu_{H}\co (H\ot \Pi^{L}_{H})) \ot H)\co (H\ot c_{H,H})\co ((\delta_{H}\co \eta_{H})\ot
H)  $

\item[ ]$= ((((\varepsilon_{H}\co
\mu_{H})\ot H)\co (H\ot c_{H,H})\co (\delta_{H}\ot
H))\ot H)\co (H\ot c_{H,H})\co ((\delta_{H}\co \eta_{H})\ot H) $

\item[ ]$=\delta_{H}\circ \Pi^{L}_{H}.$

\end{itemize}

The first and the third  equalities follow by (\ref{delta-pi-l}) and  the second one by the naturality of $c$. Finally, the last one relies on the naturality of $c$ and the coassociativity of $\delta_{H}$.

\end{proof}

\begin{remark}
\label{pi-delta-mu-pi-var}
By the equalities contained in Remark  \ref{Remark-delta-idemp} and using similar arguments to the ones 
utilized in the previous Proposition we have that 
\begin{equation}
\label{pi-delta-mu-pi-1-var}
\overline{\Pi}^{L}_{H}\circ \mu_{H}\circ (H\ot
\Pi^{L}_{H})=\overline{\Pi}^{L}_{H}\circ \mu_{H},
\end{equation}
\begin{equation}
\label{pi-delta-mu-pi-2-var}
\overline{\Pi}^{R}_{H}\circ
\mu_{H}\circ (\Pi^{R}_{H}\ot H)=\overline{\Pi}^{R}_{H}\circ \mu_{H},
\end{equation}
\begin{equation}
\label{pi-delta-mu-pi-3-var}
(\Pi^{R}_{H}\ot H)\circ \delta_{H}\circ
\overline{\Pi}^{L}_{H}=\delta_{H}\circ \overline{\Pi}^{L}_{H},
\end{equation}
\begin{equation}
\label{pi-delta-mu-pi-4-var}
( H\ot \Pi^{L}_{H})\circ \delta_{H}\circ \overline{\Pi}^{R}_{H}=\delta_{H}\circ \overline{\Pi}^{R}_{H}.
\end{equation}
\end{remark}

\begin{proposition}
\label{pi-composition}
Let $H$ be a weak Hopf quasigroup. The following identities hold:
\begin{equation}
\label{pi-composition-1}
\Pi_{H}^{L}\co
\overline{\Pi}_{H}^{L}=\Pi_{H}^{L},\;\;\;\Pi_{H}^{L}\co
\overline{\Pi}_{H}^{R}=\overline{\Pi}_{H}^{R},
\end{equation}
\begin{equation}
\label{pi-composition-2}
\overline{\Pi}_{H}^{L}\co
\Pi_{H}^{L}=\overline{\Pi}_{H}^{L},\;\;\;\overline{\Pi}_{H}^{R}\co
\Pi_{H}^{L}=\Pi_{H}^{L},
\end{equation}
\begin{equation}
\label{pi-composition-3}
\Pi_{H}^{R}\co
\overline{\Pi}_{H}^{L}=\overline{\Pi}_{H}^{L},\;\;\;
\Pi_{H}^{R}\co
\overline{\Pi}_{H}^{R}=\Pi_{H}^{R},
\end{equation}
\begin{equation}
\label{pi-composition-4}
\overline{\Pi}_{H}^{L}\co
\Pi_{H}^{R}=\Pi_{H}^{R},\;\;\; \overline{\Pi}_{H}^{R}\co
\Pi_{H}^{R}=\overline{\Pi}_{H}^{R}.
\end{equation}

\end{proposition}

\begin{proof}
We only check (\ref{pi-composition-1}) and (\ref{pi-composition-2}). The proof for the other equalities
can be verified in a similar way. Taking into account the equalities (\ref{new-pi-1}) and (a2) of Definition \ref{Weak-Hopf-quasigroup} we have

\begin{itemize}
\item[ ]$\hspace{0.38cm}\Pi_{H}^{L}\co \overline{\Pi}_{H}^{L} $

\item[ ]$=(H\ot (\varepsilon_{H}\co \mu_{H})\ot
(\varepsilon_{H}\co \mu_{H}))\co ((c^{-1}_{H,H}\co
\delta_{H}\co \eta_{H})\ot (\delta_{H}\co \eta_{H})\ot H)  $

\item[ ]$=(H\ot (\varepsilon_{H}\co \mu_{H}))\co
((c^{-1}_{H,H}\co \delta_{H}\co \eta_{H})\ot (\mu_{H}\co
(\eta_{H}\ot H)))   $

\item[ ]$=\Pi_{H}^{L}, $
\end{itemize}

and  composing with $\varepsilon_{H}\ot H$ in (\ref{pi-delta-mu-pi-4-var}) we obtain 
 $\Pi_{H}^{L}\co \overline{\Pi}_{H}^{R}=\overline{\Pi}_{H}^{R}.$ Also, composing with $\eta_{H}\ot H$ in (\ref{pi-delta-mu-pi-1-var}) we have the equality $\overline{\Pi}_{H}^{L} \co \Pi_{H}^{L} =\overline{\Pi}_{H}^{L}. $

Finally, by the usual arguments

\begin{itemize}
\item[ ]$\hspace{0.38cm} \overline{\Pi}_{H}^{R}\co \Pi_{H}^{L} $

\item[ ]$=((\varepsilon_{H}\co \mu_{H}\co c^{-1}_{H,H})\ot
(\varepsilon_{H}\co \mu_{H})\ot H)\co (H\ot (\delta_{H}\co
\eta_{H})\ot (\delta_{H}\co \eta_{H}))$

\item[ ]$=((\varepsilon_{H}\co \mu_{H}\co c^{-1}_{H,H})\ot
((\varepsilon_{H}\ot H)\co \delta_{H}))\co (H\ot (\delta_{H}\co
\eta_{H})) $

\item[ ]$=\Pi_{H}^{L}. $
\end{itemize}

\end{proof}

\begin{proposition}
\label{pi-antipode-composition}
Let $H$ be a weak Hopf quasigroup. The following identities hold: 
\begin{equation}
\label{pi-antipode-composition-1}
\Pi_{H}^{L}\co \lambda_{H}=\Pi_{H}^{L}\co
\Pi_{H}^{R}= \lambda_{H}\co \Pi_{H}^{R},
\end{equation}
\begin{equation}
\label{pi-antipode-composition-2}
\Pi_{H}^{R}\co
\lambda_{H}=\Pi_{H}^{R}\co \Pi_{H}^{L}= \lambda_{H}\co
\Pi_{H}^{L},
\end{equation}
\begin{equation}
\label{pi-antipode-composition-3}
\Pi_{H}^{L}=\overline{\Pi}_{H}^{R}\co
\lambda_{H}=\lambda_{H}
\co\overline{\Pi}_{H}^{L},
\end{equation}
\begin{equation}
\label{pi-antipode-composition-4}
\Pi_{H}^{R}=
\overline{\Pi}_{H}^{L}\co \lambda_{H}=\lambda_{H} \co
\overline{\Pi}_{H}^{R}.
\end{equation}

\end{proposition}

\begin{proof}
As in the previous Proposition, it is sufficient to check  (\ref{pi-antipode-composition-1}) and (\ref{pi-antipode-composition-3}). The proof for the other equalities can be verified in a similar way.

The equalities of  (\ref{pi-antipode-composition-1}) hold because by (a4-3) of Definition \ref{Weak-Hopf-quasigroup} and 
(\ref{pi-delta-mu-pi-1}) we have 
$$\Pi_{H}^{L}\co \lambda_{H}=\Pi_{H}^{L}\co (\lambda_{H}\ast \Pi_{H}^{L})=\Pi_{H}^{L}\co
\Pi_{H}^{R}$$
and by (\ref{pi-delta-mu-pi-4})
$$\lambda_{H}\co \Pi_{H}^{R}=(\Pi_{H}^{R}\ast \lambda_{H})\co \Pi_{H}^{R}=\Pi_{H}^{L}\co
\Pi_{H}^{R}.$$

On the other hand, by  (a4-3) of Definition \ref{Weak-Hopf-quasigroup}, (\ref{pi-delta-mu-pi-2-var}) and (\ref{pi-composition-2}) we obtain 
$$\overline{\Pi}_{H}^{R}\co
\lambda_{H}=\overline{\Pi}_{H}^{R}\co (\Pi_{H}^{R}\ast \lambda_{H})=\overline{\Pi}_{H}^{R}\co \Pi_{H}^{L}=\Pi_{H}^{L}.
$$
 Moreover, by (\ref{pi-delta-mu-pi-3-var}) and (\ref{pi-composition-1})
$$\lambda_{H}\co \overline{\Pi}_{H}^{L}=(\Pi_{H}^{R}\ast \lambda_{H})\co \overline{\Pi}_{H}^{L}=\Pi_{H}^{L}\co \overline{\Pi}_{H}^{L}=\Pi_{H}^{L},$$
and then (\ref{pi-antipode-composition-3}) holds.
\end{proof}

\begin{proposition}
\label{magma-comonad} Let $H$ be a weak  Hopf quasigroup. Put $H_{L}=Im(\Pi_{H}^{L})$ and let
$p_{L}:H\rightarrow H_{L}$ and $i_{L}:H_{L}\rightarrow H$ be the
morphisms such that $\Pi_{H}^{L}=i_{L}\co p_{L}$ and $p_{L}\co
i_{L}=id_{H_{L}}$. Then,
$$
\setlength{\unitlength}{3mm}
\begin{picture}(30,4)
\put(3,2){\vector(1,0){4}} \put(11,2.5){\vector(1,0){10}}
\put(11,1.5){\vector(1,0){10}} \put(1,2){\makebox(0,0){$H_{L}$}}
\put(9,2){\makebox(0,0){$H$}} \put(24,2){\makebox(0,0){$H\ot H$}}
\put(5.5,3){\makebox(0,0){$i_{L}$}} \put(16,3.5){\makebox(0,0){$
\delta_{H}$}} \put(16,0.5){\makebox(0,0){$(H\ot \Pi_{H}^{L}) \co
\delta_{H}$}}
\end{picture}
$$
is an equalizer diagram  and
$$
\setlength{\unitlength}{1mm}
\begin{picture}(101.00,10.00)
\put(20.00,8.00){\vector(1,0){25.00}}
\put(20.00,4.00){\vector(1,0){25.00}}
\put(55.00,6.00){\vector(1,0){21.00}}
\put(32.00,11.00){\makebox(0,0)[cc]{$\mu_{H}$ }}
\put(33.00,0.00){\makebox(0,0)[cc]{$\mu_{H}\co (H\ot \Pi_{H}^{L})
$ }} \put(65.00,9.00){\makebox(0,0)[cc]{$p_{L}$ }}
\put(13.00,6.00){\makebox(0,0)[cc]{$ H\otimes H$ }}
\put(50.00,6.00){\makebox(0,0)[cc]{$ H$ }}
\put(83.00,6.00){\makebox(0,0)[cc]{$H_{L} $ }}
\end{picture}
$$
is a coequalizer diagram. As a consequence, $(H_{L},
\eta_{H_{L}}=p_{L}\co \eta_{H}, \mu_{H_{L}}=p_{L}\co \mu_{H}\co
(i_{L}\ot i_{L}))$ is a unital magma in ${\mathcal C}$ and $(H_{L},
\varepsilon_{H_{L}}=\varepsilon_{H}\co i_{L}, \delta_{H}=(p_{L}\ot
p_{L})\co \delta_{H}\co i_{L})$ is a comonoid in ${\mathcal C}.$

\end{proposition}

\begin{proof}
 Composing with $i_{L}$ in the  equality (\ref{pi-delta-mu-pi-3}), we have that $i_{L}$ equalizes
$\delta_{H}$ and $(H\ot \Pi_{H}^{L}) \co \delta_{H}$. Now, let
$t:B\rightarrow H$ be a morphism such that $(H\ot \Pi_{H}^{L}) \co
\delta_{H}\co t= \delta_{H}\co t$. If $v=p_{L} \co t$, since
$\Pi_{H   }^{L}\co t=t$  we have $i_{L}\co v=t$. Trivially the
morphism $v$ is unique and therefore, the diagram is an equalizer
diagram. In a similar way we can prove that the second diagram is
a coaqualizer diagram using (\ref{pi-delta-mu-pi-1}) instead of (\ref{pi-delta-mu-pi-3}). Finally, note that the morphisms
$\eta_{H_{L}}$ and $\mu_{H_{L}}$ are the factorizations, through
the equalizer $i_{L}$, of the morphisms $\eta_{H}$ and $\mu_{H}\co
(i_{L}\ot i_{L})$ and then it is an easy exercise to show that
$(H_{L}, \eta_{H_{L}}, \mu_{H_{L}})$ is a unital magma in ${\mathcal
C}.$ The proof for the comonoid structure it is similar and we
leave it to the reader.
\end{proof}

\begin{example}
If $H$ is the weak Hopf quasigroup defined in Example \ref{main-example} note that 
$H_{L}=\langle [1_{x}], \; x\in {\mathcal B}_{0}\rangle$. Then, in this case we have that $H_{L}$ is a monoid because its induced product $\mu_{H_{L}}$ is associative because $[1_{x}].([1_{y}].[1_{z}])$ and $([1_{x}].[1_{y}]).[1_{z}]$ are equal to 
$[1_{x}]$ if $x=y=z$ and $0$ otherwise.

Note that if we denote by $H_{R}=Im(\Pi_{H}^{R})$ and 
$p_{R}:H\rightarrow H_{R}$ and $i_{R}:H_{R}\rightarrow H$ are the
morphisms such that $\Pi_{H}^{R}=i_{R}\co p_{R}$ and $p_{R}\co
i_{R}=id_{H_{R}}$, in this example $H_{L}=H_{R}$.
\end{example}

\begin{remark}
\label{more}  By  the second equality of 
(\ref{pi-composition-2}) it is easy to show that
$$
\setlength{\unitlength}{3mm}
\begin{picture}(30,4)
\put(3,2){\vector(1,0){4}} \put(11,2.5){\vector(1,0){10}}
\put(11,1.5){\vector(1,0){10}} \put(1,2){\makebox(0,0){$H_{L}$}}
\put(9,2){\makebox(0,0){$H$}} \put(24,2){\makebox(0,0){$H\otimes
H$}} \put(5.5,3){\makebox(0,0){$i_{L}$}}
\put(16,3.5){\makebox(0,0){$\delta_{H}$}}
\put(16,0.5){\makebox(0,0){$(H\ot \overline{\Pi}_{H}^{R})\co
\delta_{H}$}}
\end{picture}
$$
is an equalizer diagram in ${\mathcal C}$. Analogously,  by (\ref{pi-composition-1}), 
\vspace{0.1cm}
$$
\setlength{\unitlength}{1mm}
\begin{picture}(101.00,10.00)
\put(20.00,8.00){\vector(1,0){25.00}}
\put(20.00,4.00){\vector(1,0){25.00}}
\put(55.00,6.00){\vector(1,0){21.00}}
\put(32.00,11.00){\makebox(0,0)[cc]{$\mu_{H}$ }}
\put(33.00,0.00){\makebox(0,0)[cc]{$\mu_{H}\co (H\otimes
\overline{\Pi}_{H}^{L})  $ }}
\put(65.00,9.00){\makebox(0,0)[cc]{$p_{L}$ }}
\put(13.00,6.00){\makebox(0,0)[cc]{$ H\otimes H$ }}
\put(50.00,6.00){\makebox(0,0)[cc]{$ H$ }}
\put(83.00,6.00){\makebox(0,0)[cc]{$H_{L} $ }}
\end{picture}
$$
\vspace{0.1cm} is a coequalizer diagram in ${\mathcal C}$.

Also, by a similar proof to the one used in Proposition
\ref{magma-comonad}, we obtain that $H_{R}=Im(\Pi_{H}^{R})$ is a unital magma 
 in ${\mathcal C}$ with structure $(H_{R},
\eta_{H_{R}}=p_{R}\co \eta_{H}, \mu_{H_{R}}=p_{R}\co \mu_{H}\co
(i_{R}\ot i_{R}))$ and it is a comonoid in ${\mathcal C}$ where
$\varepsilon_{H_{R}}=\varepsilon_{H}\co
i_{R}$ and $\delta_{H_{R}}=(p_{R}\ot p_{R})\co \delta_{H}\co i_{R}$.  Moreover,
$$
\setlength{\unitlength}{3mm}
\begin{picture}(30,4)
\put(3,2){\vector(1,0){4}} \put(11,2.5){\vector(1,0){10}}
\put(11,1.5){\vector(1,0){10}} \put(1,2){\makebox(0,0){$H_{R}$}}
\put(9,2){\makebox(0,0){$H$}} \put(24,2){\makebox(0,0){$H\ot H$}}
\put(5.5,3){\makebox(0,0){$i_{R}$}} \put(16,3.5){\makebox(0,0){$
\delta_{H}$}} \put(16,0){\makebox(0,0){$(\Pi_{H}^{R}\ot H) \co
\delta_{H}$}}
\end{picture}
$$

$$
\setlength{\unitlength}{3mm}
\begin{picture}(30,4)
\put(3,2){\vector(1,0){4}} \put(11,2.5){\vector(1,0){10}}
\put(11,1.5){\vector(1,0){10}} \put(1,2){\makebox(0,0){$H_{R}$}}
\put(9,2){\makebox(0,0){$H$}} \put(24,2){\makebox(0,0){$H\ot H$}}
\put(5.5,3){\makebox(0,0){$i_{R}$}} \put(16,3.5){\makebox(0,0){$
\delta_{H}$}} \put(16,0){\makebox(0,0){$(\overline{\Pi}_{H}^{L}\ot
H) \co \delta_{H}$}}
\end{picture}
$$
are equalizer diagrams  and
$$
\setlength{\unitlength}{1mm}
\begin{picture}(101.00,10.00)
\put(20.00,8.00){\vector(1,0){25.00}}
\put(20.00,4.00){\vector(1,0){25.00}}
\put(55.00,6.00){\vector(1,0){21.00}}
\put(32.00,11.00){\makebox(0,0)[cc]{$\mu_{H}$ }}
\put(33.00,0.00){\makebox(0,0)[cc]{$\mu_{H}\co (\Pi_{H}^{R}\ot H)
$ }} \put(65.00,9.00){\makebox(0,0)[cc]{$p_{R}$ }}
\put(13.00,6.00){\makebox(0,0)[cc]{$ H\otimes H$ }}
\put(50.00,6.00){\makebox(0,0)[cc]{$ H$ }}
\put(83.00,6.00){\makebox(0,0)[cc]{$H_{R} $ }}
\end{picture}
$$

$$
\setlength{\unitlength}{1mm}
\begin{picture}(101.00,10.00)
\put(20.00,8.00){\vector(1,0){25.00}}
\put(20.00,4.00){\vector(1,0){25.00}}
\put(55.00,6.00){\vector(1,0){21.00}}
\put(32.00,11.00){\makebox(0,0)[cc]{$\mu_{H}$ }}
\put(33.00,0.00){\makebox(0,0)[cc]{$\mu_{H}\co
(\overline{\Pi}_{H}^{R}\ot H) $ }}
\put(65.00,9.00){\makebox(0,0)[cc]{$p_{R}$ }}
\put(13.00,6.00){\makebox(0,0)[cc]{$ H\otimes H$ }}
\put(50.00,6.00){\makebox(0,0)[cc]{$ H$ }}
\put(83.00,6.00){\makebox(0,0)[cc]{$H_{R} $ }}
\end{picture}
$$
are coequalizer diagrams.
\end{remark}

\begin{proposition}
\label{mu-assoc}
Let $H$ be a weak Hopf quasigroup. The following identities hold:
\begin{equation}
\label{mu-assoc-1}
\mu_{H}\co (\mu_{H}\ot H)\co (H\ot ((\Pi_{H}^{L}\ot H)\co \delta_{H}))=\mu_{H}=
\mu_{H}\co (\mu_{H}\ot \Pi_{H}^{R})\co (H\ot  \delta_{H}),
\end{equation}
\begin{equation}
\label{mu-assoc-2}
\mu_{H}\co (\Pi_{H}^{L}\ot \mu_{H})\co (\delta_{H}\ot H)=\mu_{H}=
\mu_{H}\co (H\ot (\mu_{H}\co ( \Pi_{H}^{R}\ot H)))\co (\delta_{H}\ot H),
\end{equation}
\begin{equation}
\label{mu-assoc-3}
\mu_{H}\co (\lambda_{H}\ot (\mu_{H}\co ( \Pi_{H}^{L}\ot H)))\co (\delta_{H}\ot H)=\mu_{H}\co (\lambda_{H}\ot H)\end{equation}
$$=
\mu_{H}\co (\Pi_{H}^{R}\ot (\mu_{H}\co ( \lambda_{H}\ot H)))\co (\delta_{H}\ot H),
$$
\begin{equation}
\label{mu-assoc-4}
\mu_{H}\co (\mu_{H}\ot H)\co (H\ot ((\lambda_{H}\ot\Pi_{H}^{L})\co \delta_{H}))=\mu_{H}\co (H\ot \lambda_{H})\end{equation}
$$= \mu_{H}\co (\mu_{H}\ot H)\co (H\ot ((\Pi_{H}^{R}\ot \lambda_{H})\co \delta_{H})).$$

\end{proposition}

\begin{proof}
Let us first prove (\ref{mu-assoc-1}). By (a1) of Definition \ref{Weak-Hopf-quasigroup} and (\ref{mu-pi-l}) we have 
$$\mu_{H}=(\varepsilon_{H}\ot H)\co \delta_{H}\co \mu_{H}=((\varepsilon_{H}\co \mu_{H})\ot \mu_{H})\co \delta_{H\ot H}=\mu_{H}\co (\mu_{H}\ot H)\co (H\ot ((\Pi_{H}^{L}\ot H)\co \delta_{H})).$$ 

On the other hand, by the coassociativity of $\delta_{H}$, (a4-6) and  (a4-7) of Definition \ref{Weak-Hopf-quasigroup}, 
\begin{itemize}
\item[ ]$\hspace{0.38cm}  \mu_{H}\co (\mu_{H}\ot \Pi_{H}^{R})\co (H\ot  \delta_{H}) $

\item[ ]$=\mu_{H}\co (\mu_{H}\ot H)\co (\mu_{H}\ot \lambda_{H}\ot H)\co (H\ot \ot H\ot \delta_{H})\co (H\ot \delta_{H})  $

\item[ ]$= \mu_{H}\co (\mu_{H}\ot H)\co (\mu_{H}\ot \lambda_{H}\ot H)\co (H\ot  \delta_{H}\ot H)\co (H\ot \delta_{H})$

\item[ ]$=\mu_{H}\co (\mu_{H}\ot H)\co (H\ot ((\Pi_{H}^{L}\ot H)\co \delta_{H}))$

\end{itemize}

The proof for (\ref{mu-assoc-2}) is similar but we must  use (\ref{mu-pi-r}), (a4-5) and  (a4-4) of Definition \ref{Weak-Hopf-quasigroup} instead of  (\ref{mu-pi-l}), (a4-6) and (a4-7) respectively.

To prove (\ref{mu-assoc-3}) first note that by (a4-5), (a4-4) of Definition \ref{Weak-Hopf-quasigroup} and the coassociativity of $\delta_{H}$ we have:
\begin{itemize}
\item[ ]$\hspace{0.38cm} \mu_{H}\co (\lambda_{H}\ot (\mu_{H}\co ( \Pi_{H}^{L}\ot H)))\co (\delta_{H}\ot H)  $

\item[ ]$= \mu_{H}\co (\lambda_{H}\ot (\mu_H\circ (H\ot \mu_H)\circ (H\ot \lambda_H\ot
H)\circ (\delta_H\ot H)))\co (\delta_{H}\ot H) $

\item[ ]$= \mu_{H}\co (\lambda_{H}\ot \mu_{H})\co (\delta_{H}\ot H)\co (H\ot (\mu_{H}\co (\lambda_{H}\ot H)))\co (\delta_{H}\ot H)$

\item[ ]$=\mu_{H}\co (\Pi_{H}^{R}\ot (\mu_{H}\co ( \lambda_{H}\ot H)))\co (\delta_{H}\ot H).$
\end{itemize}

On the other hand, by (\ref{delta-pi-l}), the naturality of $c$, the coassociativity of  $\delta_{H}$ and (a1) of Definition \ref{Weak-Hopf-quasigroup} we also have the following identity:
\begin{equation}
\label{2-mu-delta-pi-l}
(\mu_{H}\ot (\mu_{H}\co (H\ot \Pi_{H}^{L})))\co \delta_{H\ot H}=(\mu_{H}\ot H)\co (H\ot c_{H,H})\co (\delta_{H}\ot H).
\end{equation}
 
Therefore, 

\begin{itemize}
\item[ ]$\hspace{0.38cm} \mu_{H}\co (\Pi_{H}^{R}\ot (\mu_{H}\co ( \lambda_{H}\ot H)))\co (\delta_{H}\ot H)  $

\item[ ]$= (H\ot (\varepsilon_{H}\co \mu_{H}))\co (c_{H,H}\ot H)\co  (H\ot ((\delta_{H}\co \mu_{H})\co ( \lambda_{H}\ot H)))\co (\delta_{H}\ot H)  $

\item[ ]$=(H\ot (\varepsilon_{H}\co \mu_{H}))\co (c_{H,H}\ot H)\co  (H\ot (((\mu_{H}\ot \mu_{H})\co \delta_{H\ot H})\co ( \lambda_{H}\ot H)))\co (\delta_{H}\ot H) $

\item[ ]$=(H\ot (\varepsilon_{H}\co \mu_{H}))\co (c_{H,H}\ot H)\co  (H\ot (((\mu_{H}\ot (\Pi_{H}^{L}\co\mu_{H}))\co \delta_{H\ot H})\co ( \lambda_{H}\ot H)))\co (\delta_{H}\ot H)$

\item[ ]$=(H\ot (\varepsilon_{H}\co \mu_{H}))\co (c_{H,H}\ot H)\co  (H\ot (((\mu_{H}\ot (\Pi_{H}^{L}\co\mu_{H}\co (H\ot \Pi_{H}^{L})))\co \delta_{H\ot H})\co ( \lambda_{H}\ot H))) \co (\delta_{H}\ot H)$

\item[ ]$=(H\ot (\varepsilon_{H}\co \mu_{H}))\co (c_{H,H}\ot H)\co  (H\ot (((\mu_{H}\ot (\mu_{H}\co (H\ot \Pi_{H}^{L})))\co \delta_{H\ot H})\co ( \lambda_{H}\ot H))) \co (\delta_{H}\ot H)$

\item[ ]$=(H\ot (\varepsilon_{H}\co \mu_{H}))\co (c_{H,H}\ot H)\co  (H\ot (((\mu_{H}\ot H)\co (H\ot c_{H,H})\co (\delta_{H}\ot H))\co ( \lambda_{H}\ot H)))   \co (\delta_{H}\ot H)$

\item[ ]$=\mu_{H}\co (((H\ot (\varepsilon_{H}\co \mu_{H}))\co (c_{H,H}\ot H)\co  (H\ot \delta_{H} ))\ot H)\co (H\ot \lambda_{H}\ot H)   \co (\delta_{H}\ot H)$

\item[ ]$=\mu_{H}\co ((\Pi_{H}^{R}\ast \lambda_{H})\ot H) $

\item[ ]$= \mu_{H}\co (\lambda_{H}\ot H).$

\end{itemize}

In the last calculus the first and the eight equalities are consequence of (\ref{mu-pi-r}). In the second one we used (a1) of Definition \ref{Weak-Hopf-quasigroup}. The third and the fifth ones follow by (\ref{mu-pi-l-varep}), the fourth one by (\ref{pi-delta-mu-pi-1}) and the sixth one by  (\ref{2-mu-delta-pi-l}). The seventh one relies on the naturality of $c$ and the last one follows by (a4-3) of Definition \ref{Weak-Hopf-quasigroup}. 

By a similar reasoning but using that 
\begin{equation}
\label{2-mu-delta-pi-r}
((\mu_{H}\co (\Pi_{H}^{R}\ot H))\ot \mu_{H})\co \delta_{H\ot H}=(H\ot \mu_{H})\co (c_{H,H}\ot H)\co (H\ot \delta_{H})
\end{equation}
instead of (\ref{2-mu-delta-pi-l}) we obtain that (\ref{mu-assoc-4}) holds.

\end{proof}

\begin{remark}
 Note that as a consequence of (\ref{mu-assoc-1}) or (\ref{mu-assoc-2}) we have 
 \begin{equation}
\label{mu-assoc-5}
\Pi_{H}^{L}\ast id_{H}=id_{H}=id_{H}\ast \Pi_{H}^{R}.
\end{equation}
\end{remark}

\begin{proposition}
\label{mu-delta-anti}
Let $H$ be a weak Hopf quasigroup. The following identities hold:
\begin{equation}
\label{mu-delta-anti-1}
\mu_{H}\co (\Pi_{H}^{L}\ot \Pi_{H}^{R})=\mu_{H}\co c_{H,H}^{-1}\co (\Pi_{H}^{L}\ot \Pi_{H}^{R}),
\end{equation}
\begin{equation}
\label{mu-delta-anti-2}
(\Pi_{H}^{L}\ot \Pi_{H}^{R})\co \delta_{H}= (\Pi_{H}^{L}\ot \Pi_{H}^{R})\co c_{H,H}^{-1}\co \delta_{H},
\end{equation}
\begin{equation}
\label{mu-delta-anti-3}
\mu_{H}\co (\Pi_{H}^{R}\ot \Pi_{H}^{L})=\mu_{H}\co c_{H,H}\co (\Pi_{H}^{R}\ot \Pi_{H}^{L}),
\end{equation}
\begin{equation}
\label{mu-delta-anti-4}
(\Pi_{H}^{R}\ot \Pi_{H}^{L})\co \delta_{H}= (\Pi_{H}^{R}\ot \Pi_{H}^{L})\co c_{H,H}\co \delta_{H},
\end{equation}
\end{proposition}

\begin{proof}
The equalities (\ref{mu-delta-anti-3}) and (\ref{mu-delta-anti-4}) can be obtained from (\ref{mu-delta-anti-1}) and 
(\ref{mu-delta-anti-2}) composing with $c_{H,H}$. Then we only need to prove (\ref{mu-delta-anti-1}) and 
(\ref{mu-delta-anti-2}).  Note that, by (\ref{new-pi-1}), (\ref{new-pi-2}) and (a3) of Definition \ref{Weak-Hopf-quasigroup} we have: 
\begin{itemize}
\item[ ]$\hspace{0.38cm} \mu_{H}\co (\Pi_{H}^{L}\ot \Pi_{H}^{R})  $

\item[ ]$=   ((\varepsilon_{H}\co \mu_{H}\co c_{H,H}^{-1})\ot H\ot (\varepsilon_{H}\co 
\mu_{H}\co c_{H,H}^{-1}))\co (H\ot ((H\ot
\mu_{H}\ot H)\co ((\delta_{H}\co \eta_{H}) \ot (\delta_{H}\co
\eta_{H})))\ot H)$

\item[ ]$= ((\varepsilon_{H}\co \mu_{H}\co c_{H,H}^{-1})\ot H\ot (\varepsilon_{H}\co 
\mu_{H}\co c_{H,H}^{-1}))\co (H\ot ((H\ot
(\mu_{H}\co c_{H,H}^{-1})\ot H)\co ((\delta_{H}\co \eta_{H}) \ot (\delta_{H}\co
\eta_{H})))\ot H)$

\item[ ]$= \mu_{H}\co c_{H,H}^{-1}\co (\Pi_{H}^{L}\ot \Pi_{H}^{R}).$
\end{itemize}
Therefore, (\ref{mu-delta-anti-1}) holds. The proof for (\ref{mu-delta-anti-2}) is similar using (a2) of Definition \ref{Weak-Hopf-quasigroup} instead of (a3).

\end{proof}

\begin{theorem}
\label{anti-antipode}
Let $H$ be a weak Hopf quasigroup. The antipode of $H$ is antimultiplicative and anticomultiplicative, i.e. the following equalities hold:
\begin{equation}
\label{anti-antipode-1}
\lambda_{H}\co \mu_{H}=\mu_{H}\co c_{H,H}\co (\lambda_{H}\ot
\lambda_{H}),
\end{equation}
\begin{equation}
\label{anti-antipode-2}
\delta_{H}\co \lambda_{H}=(\lambda_{H}\ot \lambda_{H})\co
c_{H,H}\co \delta_{H},
\end{equation}
\end{theorem}
\begin{proof}
We will  prove (\ref{anti-antipode-1}). The proof for (\ref{anti-antipode-2}) is similar and we leave the details to the reader. 

\begin{itemize}
\item[ ]$\hspace{0.38cm}  \lambda_{H}\co \mu_{H} $

\item[ ]$=(\lambda_{H}\ast \Pi_{H}^{L})\co \mu_{H}  $

\item[ ]$=\mu_{H}\co (( \lambda_{H}\co \mu_{H})\ot (\Pi_{H}^{L}\co \mu_{H}))\co \delta_{H\ot H} $

\item[ ]$= \mu_{H}\co (( \lambda_{H}\co \mu_{H})\ot (\Pi_{H}^{L}\co \mu_{H}\co (H\ot \Pi_{H}^{L}))\co \delta_{H\ot H}$

\item[ ]$= \mu_{H}\co (\lambda_{H}\ot \Pi_{H}^{L})\co (\mu_{H}\ot H)\co (H\ot c_{H,H})\co (\delta_{H}\ot H)$

\item[ ]$=\mu_{H}\co (\mu_{H}\ot \lambda_{H})\co (\lambda_{H}\ot \delta_{H})\co (\mu_{H}\ot H)\co (H\ot c_{H,H})\co (\delta_{H}\ot H)  $

\item[ ]$= \mu_{H}\co ((\mu_{H}\co (\lambda_{H}\ot H))\ot \lambda_{H})\co (((\mu_{H}\ot H)\co (H\ot c_{H,H})\co (\delta_{H}\ot H))\ot H)\co (H\ot c_{H,H}) \co (\delta_{H}\ot H)$

\item[ ]$=\mu_{H}\co ((\mu_{H}\co (\lambda_{H}\ot H))\ot \lambda_{H})\co (((\mu_{H}\ot (\mu_{H}\co (H\ot \Pi_{H}^{L})))\co \delta_{H\ot H})\ot H) \co (H\ot c_{H,H}) \co (\delta_{H}\ot H)$

\item[ ]$= \mu_{H}\co ((\mu_{H}\co (\lambda_{H}\ot H))\ot \lambda_{H})\co (((\mu_{H}\ot (\mu_{H}\co (\mu_{H}\ot \lambda_{H})\co (H\ot \delta_{H})))\co \delta_{H\ot H})\ot H)$
\item[ ]$\hspace{0.38cm}\co (H\ot c_{H,H}) \co (\delta_{H}\ot H)  $

\item[ ]$=\mu_{H}\co ((\mu_{H}\co (\lambda_{H}\ot \mu_{H})\co (\delta_{H}\ot H))\ot H)\co (\mu_{H}\ot \lambda_{H}\ot H)\co (H\ot \delta_{H}\ot \lambda_{H})\co (H\ot c_{H,H}) \co (\delta_{H}\ot H)$

\item[ ]$=\mu_{H}\co ((\mu_{H}\co (\Pi_{H}^{R}\ot H))\ot H)\co (\mu_{H}\ot \lambda_{H}\ot H)\co (H\ot \delta_{H}\ot \lambda_{H})\co (H\ot c_{H,H}) \co (\delta_{H}\ot H)$

\item[ ]$= \mu_{H}\co ((\mu_{H}\co (\Pi_{H}^{R}\ot H))\ot H)\co ((\mu_{H}\co (\Pi_{H}^{R}\ot H))\ot \lambda_{H}\ot H)\co (H\ot \delta_{H}\ot \lambda_{H})\co (H\ot c_{H,H}) \co (\delta_{H}\ot H)$

\item[ ]$= \mu_{H}\co ((\mu_{H}\co (\Pi_{H}^{R}\ot H))\ot H)\co (((H\ot (\varepsilon_{H}\co \mu_{H}))\co (c_{H,H}\ot H)\co (H\ot \delta_{H}))\ot \lambda_{H}\ot H)\co (H\ot \delta_{H}\ot \lambda_{H}) $
\item[ ]$\hspace{0.38cm} \co (H\ot c_{H,H}) \co (\delta_{H}\ot H)$

\item[ ]$= \mu_{H}\co (((\mu_{H}\co (\lambda_{H}\ot \mu_{H})\co (\delta_{H}\ot H))\ot H)\co (((H\ot (\varepsilon_{H}\co \mu_{H}))\co (c_{H,H}\ot H)\co (H\ot \delta_{H})) \ot \lambda_{H}\ot H)$
\item[ ]$\hspace{0.38cm} \co (H\ot \delta_{H}\ot \lambda_{H}) \co (H\ot c_{H,H}) \co (\delta_{H}\ot H)$

\item[ ]$=\mu_{H}\co (\mu_{H}\ot H)\co (H\ot \mu_{H}\ot H)\co (\lambda_{H}\ot ((H\ot (\varepsilon_{H}\co \mu_{H}))\co (c_{H,H}\ot H)\co (H\ot \delta_{H}))\ot  \lambda_{H}\ot\lambda_{H})$
\item[ ]$\hspace{0.38cm} \co (c_{H,H}\ot \delta_{H}\ot H)\co (H\ot \delta_{H}\ot H)\co (H\ot c_{H,H})\co (\delta_{H}\ot H)$

\item[ ]$=\mu_{H}\co (\mu_{H}\ot H)\co (H\ot \mu_{H}\ot H)\co (\lambda_{H}\ot ((\mu_{H}\co (\Pi_{H}^{R}\ot H)
))\ot  \lambda_{H}\ot\lambda_{H})\co (c_{H,H}\ot \delta_{H}\ot H)$
\item[ ]$\hspace{0.38cm} \co (H\ot \delta_{H}\ot H)\co (H\ot c_{H,H})\co (\delta_{H}\ot H)$

\item[ ]$=\mu_{H}\co (\mu_{H}\ot H)\co (\lambda_{H}\ot ((\mu_{H}\co (\Pi_{H}^{R}\ot \Pi_{H}^{L})
))\ot  \lambda_{H})\co  (c_{H,H}\ot H\ot H)\co (H\ot \delta_{H}\ot H)\co (H\ot c_{H,H})$
\item[ ]$\hspace{0.38cm}\co (\delta_{H}\ot H)$

\item[ ]$= \mu_{H}\co ( \mu_{H}\ot H)\co (\lambda_{H}\ot ((\mu_{H}\co c_{H,H}\co (\Pi_{H}^{R}\ot \Pi_{H}^{L})
))\ot  \lambda_{H})\co  (c_{H,H}\ot H\ot H)\co (H\ot \delta_{H}\ot H)$
\item[ ]$\hspace{0.38cm}\co (H\ot c_{H,H})\co (\delta_{H}\ot H)$

\item[ ]$= \mu_{H}\co (\mu_{H}\ot H)\co (H\ot \mu_{H}\ot H)\co (((\lambda_{H}\ot \Pi_{H}^{L})\co \delta_{H})\ot 
((\Pi_{H}^{R}\ot \lambda_{H})\co \delta_{H}))\co c_{H,H}$

\item[ ]$=\mu_{H}\co (\mu_{H}\ot H)\co (\lambda_{H}\ot ((\Pi_{H}^{R}\ot \lambda_{H})\co \delta_{H}))\co c_{H,H}   $

\item[ ]$=\mu_{H}\co c_{H,H}\co (\lambda_{H}\ot
\lambda_{H}).$

\end{itemize}

The first equality follows by (a4-3) of Definition \ref{Weak-Hopf-quasigroup}, the second one by (a1) of Definition \ref{Weak-Hopf-quasigroup} and  the third one by (\ref{pi-delta-mu-pi-1}). The fourth and seventh ones relies on (\ref{2-mu-delta-pi-l}). The fifth, eighth and  sixteenth ones are consequence of (a4-6) of Definition \ref{Weak-Hopf-quasigroup}. In the sixth, ninth and eighteenth equalities we used the naturality of $c$ and the equalities tenth and thirteenth follow by (a4-4) of Definition \ref{Weak-Hopf-quasigroup}. By (\ref{pi-delta-mu-pi-2}) we obtain the eleventh equality and the twelfth and  fifteenth ones  are consequence of (\ref{mu-pi-r}). The naturality of $c$ and the coassociativity of $\delta_{H}$ imply the fourteenth equality and the seventeenth one follows by (\ref{mu-delta-anti-1}). Finally, the  nineteenth equality relies on (\ref{mu-assoc-3}) and the last one on (\ref{mu-assoc-4}).
\end{proof}

A general notion of dyslexia was introduced by Pareigis in
\cite{PA3}. The following definition is the weak Hopf quasigroup version
 of dyslexia introduced by us in the weak Hopf algebra
setting \cite{IND}. As application of (\ref{anti-antipode-1}) and
(\ref{anti-antipode-2}) we obtain a generalization of the main result about
(co)dyslexia contained in \cite{IND}.

\begin{definition}
\label{dix-codix}
Let $H$ be a weak Hopf quasigroup. We will say that $H$ is  $n$-dyslectic if
$\mu_{H}\co c_{H,H}^n= \mu_{H}$ where $c_{H,H}^n=c_{H,H}\co
\stackrel{n)}{\cdots}\co  c_{H,H}$. When $c_{H,H}^n\co \delta_{H}
=\delta_{H}$ we will say that $H$ is $n$-codyslectic. 
\end{definition}

\begin{proposition}
\label{dis} Let $H$ be a weak Hopf quasigroup. If $\lambda_{H}^n=id_{H}$, then $H$
is $n$-dyslectic and $n$-codyslectic.
\end{proposition}

\begin{proof} 
By (\ref{anti-antipode-1}), we have that 
\begin{itemize}
\item[ ]$\mu_{H}=\lambda_{H}^{n}\co\mu_{H}=\lambda_{H}^{n-1} \circ
\lambda_{H}\co\mu_{H}= \lambda_{H}^{n-1}\co \mu_{H}\co
(\lambda_{H}\ot \lambda_{H})\co c_{H,H}$ \item[
]$=\lambda_{H}^{n-2}\co  \lambda_{H}\co \mu_{H}\co (\lambda_{H}\ot
\lambda_{H})\co c_{H,H} = \lambda_{H}^{n-2}\co \mu_{H}\co
(\lambda_{H}^2\ot \lambda_{H}^2)\co c_{H,H}^2= \cdots$ \item[
]$\cdots=\mu_{H}\co (\lambda_{H}^{n}\ot \lambda_{H}^{n}) \circ
c_{H,H}^{n}=\mu_{H}\co  c_{H,H}^{n}.$
\end{itemize}

Analogously, if we use (\ref{anti-antipode-2}), with a similar calculation, we obtain
that $H$ is $n$-codyslectic. 
\end{proof}

\begin{theorem}
\label{cocommutative}
If  $H$ is a weak Hopf quasigroup (co)commutative then the antipode $\lambda_{H}$
satisfies $\lambda^{2}_{H}=id_{H}$.
\end{theorem}

\begin{proof}
 If we assume that $H$ is commutative ($\mu_{H}\co c_{H,H}=\mu_{H}$, equivalently, $\mu_{H}\co c_{H,H}^{-1}=\mu_{H}$), by (\ref{new-pi-1}) $\Pi^{L}_{H}=\overline {\Pi}^{R}_{H}$. Analogously, by (\ref{new-pi-2}), $\Pi^{R}_{H}=\overline {\Pi}^{L}_{H}$. Then
$$ \lambda_{H}\circ \lambda_{H}=\lambda_{H}\circ (\lambda_{H}\ast\Pi^{L}_{H})=(\lambda_{H}\circ \lambda_{H})\ast(\lambda_{H}\circ
\Pi^{L}_{H})=(\lambda_{H}\circ \lambda_{H})\ast(\lambda_{H}\circ
\overline{\Pi}^{R}_{H})$$
$$=(\lambda_{H}\circ \lambda_{H})\ast \Pi^{R}_{H}=\mu_H\circ (\mu_H\ot H)\circ ((\lambda_{H}\circ \lambda_{H})\ot \lambda_H\ot H)
\circ (H\ot \delta_H)\co \delta_{H}$$
$$=((\lambda_{H}\circ \lambda_{H})\ast \lambda_{H})\ast id_{H}= (\lambda_{H}\circ \Pi^{R}_{H})\ast id_{H}= (\lambda_{H}\circ \overline{\Pi}^{L}_{H})\ast id_{H}=\Pi^{L}_{H}\ast id_{H}=id_{H}.$$

The first equality follows by (a4-3) of Definition \ref{Weak-Hopf-quasigroup}, the second and seventh ones by (\ref{anti-antipode-1}) and the commutativity of $\mu_{H}$, the third and eight ones by the commutativity of $\mu_{H}$ and the fourth one by (\ref{pi-antipode-composition-4}). The fifth equality relies on (a4-7) of Definition \ref{Weak-Hopf-quasigroup} and the sixth one follows by the coassociativity of $\delta_{H}$.  In the ninth one we used (\ref{pi-antipode-composition-3}) and finally the last one follows by (\ref{mu-assoc-5}).

The proof for a cocommutative
weak Hopf quasigroup is similar using that $\Pi^{L}_{H}=\overline{\Pi}^{L}_{H}$
and $\Pi^{R}_{H}=\overline{\Pi}^{R}_{H}$. 
\end{proof}

\section{The fundamental theorem of Hopf modules}

In the following definition we introduce the notion of right-right $H$-Hopf module for a weak Hopf quasigroup $H$. Note that if $H$ is a Hopf quasigroup and ${\mathcal C}$ is the symmetric monoidal category ${\Bbb F}-Vect$, we get the notion defined by Brzezi\'nski in \cite{Brz}.

\begin{definition}
\label{Hopf-module}
Let $H$ be a  weak Hopf quasigroup and $M$  an object  in ${\mathcal
C}$. We say that $(M, \phi_{M}, \rho_{M})$ is a right-right $H$-Hopf module if the following axioms hold:
\begin{itemize}
\item[(c1)] The pair $(M, \rho_{M})$ is a right $H$-comodule, i.e. $\rho_{M}:M \rightarrow M\ot H$ is a morphism 
such that $(M\ot \varepsilon_{H})\co \rho_{M}=id_{M}$ and $(\rho_{M}\ot H)\co \rho_{M}=(M\ot \delta_{H})\co \rho_{M}$.
\item[(c2)] The morphism $\phi_{M}:M\ot H\rightarrow M$ satisfies:
\begin{itemize}
\item[(c2-1)] $\phi_{M}\co (M\ot \eta_{H})=id_{M}.$
\item[(c2-2)] $\rho_{M}\co \phi_{M}=(\phi_{M}\ot \mu_{H})\co (M\ot c_{H,H}	\ot H)\co (\rho_{M}\ot \delta_{H})$, i.e.
$\phi_{M}$ is a morphism of right $H$-comodules with the codiagonal coaction on $M\ot H$.
\end{itemize}
\item[(c3)] $\phi_{M}\circ (\phi_{M}\ot \lambda_H)\circ (M\ot \delta_{H})=\phi_{M}\co (M\ot \Pi_{H}^{L}).$
\item[(c4)] $\phi_{M}\circ (\phi_{M}\ot H)\circ (M\ot \lambda_H\ot H)\circ (M\ot \delta_H)=\phi_{M}\co (M\ot \Pi_{H}^{R}).$
\item[(c5)] $\phi_{M}\circ (\phi_{M}\ot H)\circ (M\ot \Pi_{H}^{L}\ot H)\circ (M\ot \delta_H)=\phi_{M}.$
\end{itemize}
\end{definition}

\begin{remark}
Obviously, if $H$ is a  weak Hopf quasigroup, the triple $(H, \phi_{H}=\mu_{H}, \rho_{H}=\delta_{H})$ is a  right-right $H$-Hopf module. Moreover,  if $(M, \phi_{M}, \rho_{M})$ is a right-right $H$-Hopf module, the axiom (c5)  is equivalent to 
\begin{equation}
\label{new-c5}
\phi_{M}\circ (\phi_{M}\ot \Pi_{H}^{R})\circ (M\ot \delta_{H})=\phi_{M}.
\end{equation}
because by (c3) and (c4) of Definition \ref{Hopf-module} we have that 
\begin{equation}
\label{new-c5-1}
\phi_{M}\circ (\phi_{M}\ot \Pi_{H}^{R})\circ (M\ot \delta_{H})=\phi_{M}\circ (\phi_{M}\ot H)\circ (M\ot \Pi_{H}^{L}\ot H)\circ (M\ot \delta_H).
\end{equation}
\end{remark}
Also, composing in (c2-2) with $M\ot \eta_{H}$ and $M\ot \varepsilon_{H}$ we have that
\begin{equation}
\label{new-c2-2-1}
\phi_{M}\circ (M\ot \Pi_{H}^{R})\circ \rho_{M}=id_{M}.
\end{equation}
Finally, by (c5) and (\ref{new-c5}) we obtain
\begin{equation}
\label{new-c5-2}
\phi_{M}\circ (\phi_{M}\ot H)\circ (M\ot \Pi_{H}^{L}\ot H)\co  (M\ot (\delta_{H}\co \eta_{H}))=id_{M},
\end{equation}
\begin{equation}
\label{new-c5-3}
\phi_{M}\circ (\phi_{M}\ot \Pi_{H}^{R})\circ (M\ot (\delta_{H}\co \eta_{H}))=id_{M}.
\end{equation}

\begin{proposition}
\label{idempotent}
Let $H$ be a weak Hopf quasigroup and $(M, \phi_{M}, \rho_{M})$  a right-right $H$-Hopf module. The endomorphism $q_{M}:=\phi_{M}\co (M\ot \lambda_{H})\co \rho_{M}:M\rightarrow M$  satisfies 
\begin{equation}
\label{idemp-1}
\rho_{M}\co q_{M}=(M\ot \Pi_{H}^{L})\co\rho_{M}\co q_{M}
\end{equation}
 and, as a consequence, is an idempotent. Moreover, if $M^{co H}$ (object of coinvariants)  is the image of $q_{M}$ and $p_{M}:M\rightarrow M^{co H}$, $i_{M}:M^{co H}\rightarrow M$ the morphisms such that $q_{M}=i_{M}\co p_{M}$ and 
 $id_{M^{co H}}=p_{M}\co i_{M}$, 
$$
\setlength{\unitlength}{3mm}
\begin{picture}(30,4)
\put(3,2){\vector(1,0){4}} \put(11,2.5){\vector(1,0){10}}
\put(11,1.5){\vector(1,0){10}} \put(1,2){\makebox(0,0){$M^{co H}$}}
\put(9,2){\makebox(0,0){$M$}} \put(24,2){\makebox(0,0){$M\ot H$}}
\put(5.5,3){\makebox(0,0){$i_{M}$}}
\put(16,3.5){\makebox(0,0){$ \rho_{M}$}}
\put(16,0.15){\makebox(0,0){$(M\ot \overline{\Pi}_{H}^{R})\co\rho_{M}$}}
\end{picture}
$$
is an equalizer diagram.
\end{proposition}

\begin{proof} 
The equality (\ref{idemp-1}) holds because
\begin{itemize}
\item[ ]$\hspace{0.38cm} \rho_{M}\co q_{M}  $

\item[ ]$= (\phi_{M}\ot \mu_{H})\co (M\ot  c_{H,H}\ot H)\co (\rho_{M}\ot (\delta_{H}\co \lambda_{H}))\co \rho_{M} $

\item[ ]$= (\phi_{M}\ot \mu_{H})\co (M\ot  c_{H,H}\ot H)\co (\rho_{M}\ot ((\lambda_{H}\ot \lambda_{H})\co c_{H,H}\co \delta_{H}))\co \rho_{M}  $

\item[ ]$= ((\phi_{M}\co (M\ot \lambda_{H}))\ot \Pi_{H}^{L})\co (M\ot (c_{H,H}\co \delta_{H}))\co \rho_{M} $

\item[ ]$=  ((\phi_{M}\co (M\ot \lambda_{H}))\ot (\Pi_{H}^{L}\co \Pi_{H}^{L}))\co (M\ot (c_{H,H}\co \delta_{H}))\co \rho_{M}$

\item[ ]$= (M\ot \Pi_{H}^{L})\co\rho_{M}\co q_{M}$

\end{itemize}
where the first equality follows by (c2-2) of Definition \ref{Hopf-module}, the second one by (\ref{anti-antipode-2}), the third one relies on (c1) of Definition \ref{Hopf-module} as well as the naturality of the braiding, the fourth one is a consequence of  the properties of $\Pi_{H}^{L}$ and the last one uses the arguments of the  three first identities but in the inverse order.

On the other hand, $q_{M}$ is an idempotent. Indeed, 
\begin{itemize}
\item[ ]$\hspace{0.38cm} q_{M}\co q_{M}  $

\item[ ]$= \phi_{M} \co (M\ot \lambda_{H})\co \rho_{M}\co q_{M}$

\item[ ]$= \phi_{M} \co (M\ot (\lambda_{H}\co \Pi_{H}^{L}) )\co \rho_{M}\co q_{M}  $

\item[ ]$=  \phi_{M} \co (M\ot (\lambda_{H}\co \overline{\Pi}_{H}^{R}\co \Pi_{H}^{L}) )\co \rho_{M}\co q_{M}$

\item[ ]$=\phi_{M} \co (M\ot  \Pi_{H}^{R} )\co \rho_{M}\co q_{M}   $

\item[ ]$=  q_{M}.$

\end{itemize}
 In the last equalities, the first one follows by definition, the second one by (\ref{idemp-1}), the third one by (\ref{pi-composition-2}), the fourth one by (\ref{pi-antipode-composition-4}) and (\ref{idemp-1}) and the last one by  (\ref{new-c2-2-1}).

Finally, by (\ref{pi-antipode-composition-4}) and (\ref{new-c2-2-1}) $\phi_{M} \co (M\ot (\lambda_{H}\co \overline{\Pi}_{H}^{R}) )\co \rho_{M}=\phi_{M} \co (M\ot \Pi_{H}^{R}) )\co \rho_{M}=id_{M}.$
Then, 

$$
\setlength{\unitlength}{3mm}
\begin{picture}(30,4)
\put(3,2){\vector(1,0){4}} \put(11,2.5){\vector(1,0){10}}
\put(11,1.5){\vector(1,0){10}} \put(1,2){\makebox(0,0){$M^{co H}$}}
\put(9,2){\makebox(0,0){$M$}} \put(24,2){\makebox(0,0){$M\ot H$}}
\put(5.5,3){\makebox(0,0){$i_{M}$}}
\put(16,3.5){\makebox(0,0){$ \rho_{M}$}}
\put(16,0.15){\makebox(0,0){$(M\ot \overline{\Pi}_{H}^{R})\co\rho_{M}$}}
\end{picture}
$$

is a split cofork  \cite{Mac} and thus an equalizer diagram.

\end{proof}

\begin{remark}
\label{idem-2}
Note that, in the conditions of Proposition (\ref{idempotent}), by (\ref{pi-composition-1}) and (\ref{pi-composition-2}), we obtain that
$$
\setlength{\unitlength}{3mm}
\begin{picture}(30,4)
\put(3,2){\vector(1,0){4}} \put(11,2.5){\vector(1,0){10}}
\put(11,1.5){\vector(1,0){10}} \put(1,2){\makebox(0,0){$M^{co H}$}}
\put(9,2){\makebox(0,0){$M$}} \put(24,2){\makebox(0,0){$M\ot H$}}
\put(5.5,3){\makebox(0,0){$i_{M}$}}
\put(16,3.5){\makebox(0,0){$ \rho_{M}$}}
\put(16,0.15){\makebox(0,0){$(M\ot\Pi_{H}^{L})\co\rho_{M}$}}
\end{picture}
$$
is also an equalizer diagram.

Moreover, by the comodule condition and (c4) of Definition \ref{Hopf-module} we have 
\begin{equation}
\label{new-c5-2-1}
\phi_{M}\circ (q_{M}\ot H)\co \rho_{M}=id_{M}.
\end{equation}

Finally, the following identities hold:
\begin{equation}
\label{new-c5-2-2}
\rho_{M}\co \phi_{M}\co (i_{M}\ot H)=(\phi_{M}\ot H)\co (i_{M}\ot \delta_{H}),
\end{equation}
\begin{equation}
\label{new-c5-2-3}
p_{M}\co \phi_{M}\co (i_{M}\ot H)=p_{M}\co \phi_{M}\co (i_{M}\ot \Pi_{H}^{L}),
\end{equation}
\begin{equation}
\label{new-c5-2-4}
p_{M}\co \phi_{M}\co (i_{M}\ot H)= p_{M}\co \phi_{M}\co (i_{M}\ot \overline{\Pi}_{H}^{L}).
\end{equation}
Indeed:

\begin{itemize}
\item[ ]$\hspace{0.38cm} \rho_{M}\co \phi_{M}\co (i_{M}\ot H)  $

\item[ ]$= (\phi_{M}\ot \mu_{H})\co (M\ot c_{H,H}\ot H)\co ((\rho_{M}\co i_{M})\ot \delta_{H})$

\item[ ]$=   (\phi_{M}\ot \mu_{H})\co (M\ot c_{H,H}\ot H)\co (((M\ot \overline{\Pi}_{H}^{R})\co \rho_{M}\co i_{M})\ot \delta_{H})$

\item[ ]$= (\phi_{M}\ot (\mu_{H}\co (\overline{\Pi}_{H}^{R}\ot H)) )\co (M\ot c_{H,H}\ot H)\co
 ((\rho_{M}\co i_{M})\ot \delta_{H})  $

\item[ ]$= (\phi_{M}\ot (\varepsilon_{H}\co \mu_{H})\ot H)\co (M\ot c_{H,H}\ot \delta_{H})\co ((\rho_{M}\co i_{M})\ot \delta_{H})  $

\item[ ]$= (((M\ot \varepsilon_{H})\co \rho_{M}\co \phi_{M})\ot H)\co (i_{M}\ot \delta_{H}) $

\item[ ]$= (\phi_{M}\ot H)\co (i_{M}\ot \delta_{H}) .$

\end{itemize}

The first equality follows from (c2-2) of Definition \ref{Hopf-module}, the second one by Proposition \ref{idempotent} and the third one by the naturality of the braiding. The fourth equality is a consequence of (\ref{mu-pi-r-var}). In the fifth one we used the coassociativity of $\delta_{H}$ and (c2-2) of Definition \ref{Hopf-module}. The last one follows by  (c1) of Definition \ref{Hopf-module}.

On the other hand, by (\ref{new-c5-2-2}) and (a4-6) of Definition \ref{Weak-Hopf-quasigroup} we have

\begin{itemize}
\item[ ]$\hspace{0.38cm} p_{M}\co \phi_{M}\co (i_{M}\ot H)  $

\item[ ]$= p_{M}\co q_{M}\co \phi_{M}\co (i_{M}\ot H)$

\item[ ]$=  p_{M}\co \phi_{M}\co (M\ot \lambda_{H})\co \rho_{M}\co \phi_{M}\co (i_{M}\ot H) $

\item[ ]$= p_{M}\co \phi_{M}\co (\phi_{M}\ot \lambda_{H})\co (i_{M}\ot \delta_{H}) $

\item[ ]$= p_{M}\co \phi_{M}\co (i_{M}\ot \Pi_{H}^{L}) .$

\end{itemize}

Finally, composing with $ M^{co H}\ot \overline{\Pi}_{H}^{L}$ in (\ref{new-c5-2-3}) and using (\ref{pi-composition-1}) we obtain (\ref{new-c5-2-4}).

\end{remark}

\begin{proposition}
\label{tensor-idempotent}
Let $H$ be a weak Hopf quasigroup, $(M, \phi_{M}, \rho_{M})$  a right-right $H$-Hopf module. The endomorphism 
$$\nabla_{M}:=(p_{M}\ot H)\co \rho_{M}\co \phi_{M}\co (i_{M}\ot H):M^{co H}\ot H\rightarrow M^{co H}\ot H$$
is an idempotent and the equalities 
\begin{equation}
\label{tensor-idempotent-1}
\nabla_{M}=((p_{M}\co \phi_{M})\ot H)\co  (i_{M}\ot \delta_{H}), 
\end{equation}
\begin{equation}
\label{tensor-idempotent-2}
(M^{co H}\ot \delta_{H})\co \nabla_{M}= (\nabla_{M}\ot H)\co (M^{co H}\ot \delta_{H}).
\end{equation}
\begin{equation}
\label{tensor-idempotent-3}
 \nabla_{M}= (M^{co H}\ot \mu_{H})\co ((\nabla_{M}\co (M^{co H}\ot \eta_{H}))\ot H).
\end{equation}
hold.

\end{proposition}

\begin{proof} Trivially, by (\ref{new-c5-2-1}) we have that 
$$\nabla_{M}\co \nabla_{M}=(p_{M}\ot H)\co \rho_{M}\co \phi_{M}\circ (q_{M}\ot H)\co \rho_{M}\co\phi_{M}\co (i_{M}\ot H)=\nabla_{M}.$$
The equality (\ref{tensor-idempotent-1}) follows from (\ref{new-c5-2-2}) and (\ref{tensor-idempotent-2}) is a consequence of (\ref{tensor-idempotent-1}) and the coassociativity of $\delta_{H}$. Finally, (\ref{tensor-idempotent-3}) holds because, by (\ref{tensor-idempotent-1}), (\ref{delta-pi-l-var}) and (\ref{new-c5-2-4}) we have
$$(M^{co H}\ot \mu_{H})\co ((\nabla_{M}\co (M^{co H}\ot \eta_{H}))\ot H)=((p_{M}\co \phi_{M}\co (M\ot \overline{\Pi}_{H}^{L}))\ot H)\co (i_{M}\ot \delta_{H})=\nabla_{M}.$$

\end{proof}

\begin{remark}
\label{Th-mod} In the conditions of Proposition \ref{tensor-idempotent} we define the morphisms
$$\omega_{M}:M^{co H}\ot H\rightarrow M,\;\;\;\;\;
\omega_{M}^{\prime}:M\rightarrow M^{co H}\ot H$$  by
$\omega_{M}=\phi_{M}\co (i_{M}\otimes H)$ and
$\omega_{M}^{\prime}=(p_{M}\otimes H)\co \rho_{M}$. Then, 
 $\omega_{M}\circ \omega_{M}^{\prime}=id_{M}$ and
$\nabla_{M}=\omega_{M}^{\prime}\co \omega_{M}$. Also, we have a commutative diagram
$$
\unitlength=1mm \special{em:linewidth 0.4pt} \linethickness{0.4pt}
\begin{picture}(64.00,37.00)
\put(10.00,20.00){\vector(1,0){30.00}}
\put(8.00,18.00){\vector(4,-3){16.00}}
\put(27.00,6.00){\vector(1,1){12.00}}
\put(8.00,22.00){\vector(3,2){16.00}}
\put(26.00,32.00){\vector(4,-3){12.00}}

\put(0.00,20.00){\makebox(0,0)[cc]{$M^{co H}\ot H$}}
\put(50.00,20.00){\makebox(0,0)[cc]{$M^{co H}\ot H$}}
\put(24.00,36.00){\makebox(0,0)[cc]{$M$}}
\put(24.00,2.00){\makebox(0,0)[cc]{$M^{co H}\times H$}}
\put(13.00,29.00){\makebox(0,0)[cc]{$\omega_{M}$}}
\put(35.00,30.00){\makebox(0,0)[cc]{$\omega_{M}^{\prime}$}}
\put(9.00,10.00){\makebox(0,0)[cc]{$p_{M^{co H}\ot H}$}}
\put(39.00,9.00){\makebox(0,0)[cc]{$i_{M^{co H}\ot H}$}}
\put(23.00,22.00){\makebox(0,0)[cc]{$\nabla_{M}$}}

\end{picture}
$$
where $M^{co H}\times H$ denotes the image of $\nabla_{M}$ and $p_{M^{co H}\ot H}$,  $i_{M^{co H}\ot H}$ are the morphisms such that $p_{M^{co H}\ot H}\co i_{M^{co H}\ot H}=id_{M^{co H}\times H}$ and $i_{M^{co H}\ot H}\co p_{M^{co H}\ot H}=\nabla_{M}$. Therefore,
the morphism $\alpha_{M}=p_{M^{co H}\ot H}\circ \omega^{\prime}_{M}$ is an
isomorphism of right $H$-modules (i.e. $\rho_{M^{co H}\times H}\co \alpha=(\alpha\ot H)\co \rho_{M}$)
with inverse $\alpha^{-1}_{M}=\omega_{M}\circ i_{M^{co H}\ot H}$. The
comodule structure of $M^{co H}\times H$ is the one induced by
the isomorphism $\alpha$ and it is equal to
$$
\rho_{M^{co H}\times H}=(p_{M^{co H}\ot H}\ot
H)\co (M^{co H}\ot \delta_{H})\co i_{M^{co H}\ot H}.
$$
\end{remark}

\begin{proposition}
\label{isomorphism}
Let $H$ be a weak Hopf quasigroup and  $(M, \phi_{M}, \rho_{M})$, $(N, \phi_{N}, \rho_{N})$ right-right $H$-Hopf modules. If there exists a right $H$-comodule isomorphism $\alpha:M\rightarrow N$, the triple $(M,\phi_{M}^{\alpha}=\alpha^{-1}\co \phi_{N}\co (\alpha\ot H), \rho_{M})$ is a right-right $H$-Hopf module.
\end{proposition}

\begin{proof} The proof follows easily because, if $\alpha$ is a right $H$-comodule isomorphism,  $\rho_{M}=(\alpha^{-1}\ot H)\co \rho_{N}\co \alpha$ holds. 
\end{proof}

\begin{proposition}
\label{structure-times}
Let $H$ be a weak Hopf quasigroup, $(M, \phi_{M}, \rho_{M})$  a right-right $H$-Hopf module. The triple $(M^{co H}\times H, \phi_{M^{co H}\times H}, \rho_{M^{co H}\ot H})$ where 
$$\phi_{M^{co H}\times H}=p_{M^{co H}\ot H}\co (M^{co H}\ot \mu_{H})\co (i_{M^{co H}\ot H}\ot H),$$
and $\rho_{M^{co H}\ot H}$ is the coaction defined in Remark \ref{Th-mod}, is a right-right $H$-Hopf module.
\end{proposition}

\begin{proof}
By Remark \ref{Th-mod} we have that $(M^{co H}\times H,  \rho_{M^{co H}\times H})$ is a right $H$-comodule and it is clear that (c2-1) of Definition \ref{Hopf-module} holds. On the other hand, by (\ref{tensor-idempotent-2}), 
$$ \rho_{M^{co H}\times H}\co  \phi_{M^{co H}\times H} 
= (p_{M^{co H}\ot H}\ot H)\co    (M^{co H}\ot (\delta_{H}\co \mu_{H}))\co 
(i_{M^{co H}\ot H}\ot H).  $$

Moreover, by (\ref{tensor-idempotent-2}), the properties of $\nabla_{M}$ and (a1) of Definition \ref{Weak-Hopf-quasigroup}
we obtain

\begin{itemize}
\item[ ]$\hspace{0.38cm} (\phi_{M^{co H}\times H}\ot \mu_{H})\co (M^{co H}\times H\ot c_{H,H}\ot H)\co (\rho_{M^{co H}\times H}\ot \delta_{H})  $

\item[ ]$= ((p_{M^{co H}\ot H}\co (M^{co H}\ot \mu_{H}))\ot H)\co (\nabla_{M}\ot H\ot \mu_{H})\co (M^{co H}\ot H\ot c_{H,H}\ot H) $
\item[ ]$\hspace{0.38cm} \co (M^{co H}\ot\delta_{H}\ot \delta_{H})\co (i_{M^{co H}\ot H}\ot H)$

\item[ ]$=(p_{M^{co H}\ot H}\ot H)\co (M^{co H}\ot  ((\mu_{H}\ot \mu_{H})\co \delta_{H\ot H}))\co ((\nabla_{M}\co
i_{M^{co H}\ot H})\ot H)  $

\item[ ]$= (p_{M^{co H}\ot H}\ot H)\co    (M^{co H}\ot (\delta_{H}\co \mu_{H}))\co 
(i_{M^{co H}\ot H}\ot H)  $

\end{itemize}
and then  (c2-2) of Definition \ref{Hopf-module} holds.

The proof for (c3) of Definition \ref{Hopf-module} is the following:
\begin{itemize}
\item[ ]$\hspace{0.38cm} \phi_{M^{co H}\times H}  \co (\phi_{M^{co H}\times H} \ot \lambda_{H})\co (M^{co H}\times H\ot \delta_{H})$

\item[ ]$=p_{M^{co H}\times H}  \co  ((p_{M}\co \phi_{M})\ot \mu_{H})\co (i_{M}\ot (\delta_{H}\co \mu_{H})\ot \lambda_{H})\co (i_{M^{co H}\times H}\ot \delta_{H})$

\item[ ]$=p_{M^{co H}\times H}  \co  ((p_{M}\co \phi_{M})\ot \mu_{H})\co (i_{M}\ot ((\mu_{H}\ot \mu_{H})\co \delta_{H\ot H})\ot \lambda_{H}) \co (i_{M^{co H}\times H}\ot \delta_{H}) $

\item[ ]$= p_{M^{co H}\times H}  \co  ((p_{M}\co \phi_{M}\co (i_{M}\ot \mu_{H}))\ot (\mu_{H}\co (H\ot \Pi_{H}^{L})))\co 
(M^{co H}\ot \delta_{H\ot H})\co  (i_{M^{co H}\times H}\ot H)$

\item[ ]$=  p_{M^{co H}\times H}  \co  ((p_{M}\co \phi_{M}\co (i_{M}\ot \mu_{H}))\ot H)\co (M^{co H}\ot H\ot c_{H,H})\co (M^{co H}\ot \delta_{H}\ot H)\co (i_{M^{co H}\times H}\ot H)$

\item[ ]$=  p_{M^{co H}\times H}  \co  ((p_{M}\co \phi_{M}\co (i_{M}\ot (\Pi_{H}^{L}\co \mu_{H})))\ot H)\co (M^{co H}\ot H\ot c_{H,H})\co (M^{co H}\ot \delta_{H}\ot H)$
\item[ ]$\hspace{0.38cm}\co (i_{M^{co H}\times H}\ot H)$

\item[ ]$=  p_{M^{co H}\times H}  \co  ((p_{M}\co \phi_{M}\co (i_{M}\ot \mu_{H}))\ot H)\co (M^{co H}\ot H\ot c_{H,H})\co (M^{co H}\ot \delta_{H}\ot H)\co (i_{M^{co H}\times H}\ot \Pi_{H}^{L})$

\item[ ]$=p_{M^{co H}\times H}  \co  ((p_{M}\co \phi_{M})\ot H)\co (i_{M}\ot (((\varepsilon_{H}\co \mu_{H})\ot H\ot H)\co (H\ot c_{H,H}\ot H) \co (\delta_{H}\ot c_{H,H}) \co (\delta_{H}\ot H)))$
\item[ ]$\hspace{0.38cm}\co (i_{M^{co H}\times H}\ot H)$

\item[ ]$=p_{M^{co H}\times H}  \co  ((p_{M}\co \phi_{M})\ot H)\co (i_{M}\ot (((\varepsilon_{H}\co \mu_{H})\ot \delta_{H})\co (H\ot c_{H,H})  \co (\delta_{H}\ot H)))\co (i_{M^{co H}\times H}\ot H)$

\item[ ]$=p_{M^{co H}\times H}  \co  ((p_{M}\co \phi_{M})\ot H)\co (i_{M}\ot (\delta_{H}\co \mu_{H}\co (H\ot \Pi_{H}^{L})))\co (i_{M^{co H}\times H}\ot H)$

\item[ ]$=p_{M^{co H}\times H}  \co \nabla_{M}\co (M^{co H} \ot (\mu_{H}\co (H\ot \Pi_{H}^{L}) ))\co (i_{M^{co H}\times H}\ot H)$

\item[ ]$= \phi_{M^{co H}\times H}  \co  (M^{co H}\times H\ot \Pi_{H}^{L})$

\end{itemize}

The first and tenth equalities follow by (\ref{tensor-idempotent-1}), the second one by (a1) of Definition \ref{Weak-Hopf-quasigroup} and the third one by the coassociativity of $\delta_{H}$ and (a4-6) of Definition \ref{Weak-Hopf-quasigroup}. In the fourth one we used (\ref{2-mu-delta-pi-l}). The fifth equality relies on (\ref{new-c5-2-3}) and the sixth one follows by (\ref{pi-delta-mu-pi-1}) and (\ref{new-c5-2-3}). The seventh one is a consequence  of the naturality of the braiding and (\ref{mu-pi-l}). The eighth equality follows by the naturality of the braiding and the coassociativity of $\delta_{H}$. Finally, the ninth equality follows by (\ref{mu-pi-l}) and  the last one relies on the properties of $\nabla_{M}$.

We continue in this fashion proving (c4) of Definition \ref{Hopf-module}. Indeed:

\begin{itemize}
\item[ ]$\hspace{0.38cm} \phi_{M^{co H}\times H}  \co (\phi_{M^{co H}\times H} \ot H)\co (M^{co H}\times H\ot \lambda_{H}\ot H)\co (M^{co H}\times H\ot \delta_{H})$

\item[ ]$=p_{M^{co H}\times H}  \co  ((p_{M}\co \phi_{M})\ot \mu_{H})\co (i_{M}\ot (\delta_{H}\co \mu_{H}\co (H\ot \lambda_{H}))\ot H)\co (i_{M^{co H}\times H}\ot \delta_{H})$

\item[ ]$=p_{M^{co H}\times H}  \co  ((p_{M}\co \phi_{M})\ot \mu_{H})\co (i_{M}\ot ((\mu_{H}\ot \mu_{H})\co \delta_{H\ot H}\co (H\ot \lambda_{H}))\ot H)\co (i_{M^{co H}\times H}\ot \delta_{H}) $

\item[ ]$= p_{M^{co H}\times H}  \co  ((p_{M}\co \phi_{M}\co (M\ot \Pi_{H}^{L}))\ot \mu_{H})\co (i_{M}\ot ((\mu_{H}\ot \mu_{H})\co \delta_{H\ot H}\co (H\ot \lambda_{H}))\ot H)\co (i_{M^{co H}\times H}\ot \delta_{H})$

\item[ ]$=p_{M^{co H}\times H}  \co  ((p_{M}\co \phi_{M})\ot \mu_{H})\co (i_{M}\ot [((\varepsilon_{H}\co \mu_{H})\ot (\varepsilon_{H}\co \mu_{H})\ot H\ot \mu_{H})\co (H\ot \delta_{H}\ot c_{H,H}\ot H\ot H) $
\item[ ]$\hspace{0.38cm}\co (H\ot c_{H,H}\ot c_{H,H}\ot H)\co ((\delta_{H}\co \eta_{H})\ot \delta_{H}\ot (\delta_{H}\co \lambda_{H}))]\ot H)\co (i_{M^{co H}\times H}\ot \delta_{H}) $

\item[ ]$=p_{M^{co H}\times H}  \co  ((p_{M}\co \phi_{M})\ot \mu_{H})\co 
(i_{M}\ot  (\Pi_{H}^{L}\ot ((\varepsilon_{H}\ot H)\co \delta_{H}\co \mu_{H}))\ot H)\co (M^{co H}\ot \delta_{H}\ot\lambda_{H}\ot H)$
\item[ ]$\hspace{0.38cm} \co (i_{M^{co H}\times H}\ot \delta_{H}) $  

\item[ ]$=p_{M^{co H}\times H}  \co  ((p_{M}\co \phi_{M})\ot \mu_{H})\co 
(i_{M}\ot   H\ot  \mu_{H}\ot H)\co (M^{co H}\ot \delta_{H}\ot\lambda_{H}\ot H)\co (i_{M^{co H}\times H}\ot \delta_{H})$

\item[ ]$=p_{M^{co H}\times H}  \co  ((p_{M}\co \phi_{M})\ot \mu_{H})\co 
(i_{M}\ot   \delta_{H}\ot  \Pi_{H}^{R})\co (i_{M^{co H}\times H}\ot H)  $

\item[ ]$= p_{M^{co H}\times H}  \co  (M^{co H}\ot \mu_{H})\co 
((\nabla_{M}\co i_{M^{co H}\times H})\ot \Pi_{H}^{R}) $

\item[ ]$= \phi_{M^{co H}\times H}  \co  (M^{co H}\times H\ot \Pi_{H}^{R})$

\end{itemize}

The first and eighth equalities follow by (\ref{tensor-idempotent-1}), the second one by (a1) of Definition \ref{Weak-Hopf-quasigroup} and the third one by (\ref{new-c5-2-3}). In the fourth one we used the naturality of the braiding and (a2) of Definition \ref{Weak-Hopf-quasigroup}. The fifth one is a consequence of the naturality of the braiding, the coassociativity of $\delta_{H}$ and (a1) of Definition \ref{Weak-Hopf-quasigroup}. The sixth equality follows from the counit properties and (\ref{new-c5-2-3}) and the seventh one by (a4-7) of Definition \ref{Weak-Hopf-quasigroup}. Finally, the last equality is a consequence of the properties of $\nabla_{M}$.

The only point remaining is (c5) of Definition \ref{Hopf-module}. This equality holds because:

\begin{itemize}
\item[ ]$\hspace{0.38cm} \phi_{M^{co H}\times H}  \co (\phi_{M^{co H}\times H} \ot H)\co (M^{co H}\times H\ot \Pi_{H}^{L}\ot H)\co (M^{co H}\times H\ot \delta_{H})$

\item[ ]$=p_{M^{co H}\times H}  \co  ((p_{M}\co \phi_{M})\ot \mu_{H})\co (i_{M}\ot (\delta_{H}\co \mu_{H}\co (H\ot \Pi_{H}^{L}))\ot H)\co (i_{M^{co H}\times H}\ot \delta_{H}) $

\item[ ]$=p_{M^{co H}\times H}  \co  ((p_{M}\co \phi_{M})\ot \mu_{H})\co (i_{M}\ot ((\mu_{H}\ot \mu_{H})\co \delta_{H\ot H}\co (H\ot \Pi_{H}^{L}))\ot H) \co (i_{M^{co H}\times H}\ot \delta_{H}) $

\item[ ]$=p_{M^{co H}\times H}  \co  ((p_{M}\co \phi_{M})\ot \mu_{H})\co (i_{M}\ot ((\mu_{H}\ot (\mu_{H}\co (H\ot \Pi_{H}^{L})))\co \delta_{H\ot H} \co (H\ot \Pi_{H}^{L}))\ot H)$
\item[ ]$\hspace{0.38cm}\co (i_{M^{co H}\times H}\ot \delta_{H}) $ 

\item[ ]$=p_{M^{co H}\times H}  \co  ((p_{M}\co \phi_{M})\ot H)\co (i_{M}\ot ( (\Pi_{H}^{L}\ot H)\co \mu_{H\ot H} \co (\delta_{H}\ot ((\Pi_{H}^{L}\ot H)\co \delta_{H}))))\co (i_{M^{co H}\times H}\ot H) $

\item[ ]$=p_{M^{co H}\times H}  \co  ((p_{M}\co \phi_{M})\ot H)\co (i_{M}\ot ( (\Pi_{H}^{L}\co \mu_{H}\co (H\ot \Pi_{H}^{L}))\ot \mu_{H})\co \delta_{H\ot H})\co (i_{M^{co H}\times H}\ot H) $  

\item[ ]$=p_{M^{co H}\times H}  \co  ((p_{M}\co \phi_{M})\ot H)\co (i_{M}\ot ( (\Pi_{H}^{L}\ot H)\co \delta_{H}\co \mu_{H})) \co (i_{M^{co H}\times H}\ot H) $  

\item[ ]$=p_{M^{co H}\times H}  \co  ((p_{M}\co \phi_{M})\ot H)\co (i_{M}\ot ( \delta_{H}\co \mu_{H})) \co (i_{M^{co H}\times H}\ot H) $ 

\item[ ]$= p_{M^{co H}\times H}  \co  \nabla_{M}\co (M^{co H}\ot \mu_{H})\co \co (i_{M^{co H}\times H}\ot H) $

\item[ ]$= \phi_{M^{co H}\times H}$

\end{itemize}

The first  equalitiy follows by (\ref{tensor-idempotent-1}), the second one by (a1) of Definition \ref{Weak-Hopf-quasigroup} and the third one by (\ref{pi-delta-mu-pi-3}). In the fourth one we used (\ref{2-mu-delta-pi-l}) as well as (\ref{new-c5-2-3}). The fifth one relies on the naturality of the braiding and the sixth one is a consequence of (\ref{pi-delta-mu-pi-1}) and (a1) of Definition \ref{Weak-Hopf-quasigroup}. The seventh one follows by (\ref{new-c5-2-3}) and  the eighth one by \ref{tensor-idempotent-1}. Finally, the last one follows by the properties of $\nabla_{M}$. 

\end{proof}

\begin{proposition}
\label{structure-times-2}
Let $H$ be a weak Hopf quasigroup, $(M, \phi_{M}, \rho_{M})$  be a right-right $H$-Hopf module and  $\alpha_{M}:M\rightarrow M^{co H}\times H$ be the isomorphism of right $H$-comodules defined in Remark \ref{Th-mod}.  Then, for the action $\phi_{M}^{\alpha_{M}}$ introduced  in Proposition \ref{isomorphism}, the triple $(M, \phi_{M}^{\alpha_{M}}, \rho_{M})$  is a right-right $H$-Hopf module  with the same object of coinvariants of  $(M, \phi_{M}, \rho_{M})$. Moreover, the identity 
\begin{equation}
\label{action-alpha}
\phi_{M}^{\alpha_{M}}=\phi_{M}\co (q_{M}\ot \mu_{H})\co (\rho_{M}\ot H)
\end{equation}
holds and $q_{M}^{\alpha_{M}}=q_{M}$ where $q_{M}^{\alpha_{M}}=\phi_{M}^{\alpha_{M}}\co (M\ot \lambda_{H})\co \rho_{M}$ is the idempotent morphism associated to the 
Hopf module $(M, \phi_{M}^{\alpha_{M}}, \rho_{M})$. Finally, if  for $(M, \phi_{M}^{\alpha_{M}}, \rho_{M})$, $\nabla_{M}^{\alpha_{M}}$ denotes the idempotent morphism defined in Proposition \ref{tensor-idempotent}, we have that 
\begin{equation}
\label{idemp-1-2}
\nabla_{M}^{\alpha_{M}}=\nabla_{M}
\end{equation} 
and then, for $(M, \phi_{M}^{\alpha_{M}}, \rho_{M})$, the associated isomorphism between $M$ and $M^{co H}\times H$ defined in Remark \ref{Th-mod} is $\alpha_{M}$. Finally, 
\begin{equation}
\label{phi-square}
(\phi_{M}^{\alpha_{M}})^{\alpha_{M}}=
\phi_{M}^{\alpha_{M}}
\end{equation}
holds.
\end{proposition}

\begin{proof}
By Proposition \ref{isomorphism} we obtain that  $(M, \phi_{M}^{\alpha_{M}}, \rho_{M})$  is a right-right $H$-Hopf module and by the equalizer diagram of Proposition \ref{idempotent} the object of coinvariants of $(M, \phi_{M}^{\alpha_{M}}, \rho_{M})$ is equal to the one of $(M, \phi_{M}, \rho_{M})$. Also, by (\ref{new-c5-2-1}) we have 
\begin{equation}
\label{nabla-i-phi}
\phi_{M}\co (i_{M}\ot H)\co \nabla_{M}=\phi_{M}\co (i_{M}\ot H)
\end{equation}
and 
\begin{equation}
\label{rho-p-phi}
 \nabla_{M}\co (p_{M}\ot H)\co \rho_{M}=(p_{M}\ot H)\co \rho_{M}.
\end{equation}
Then, (\ref{action-alpha}) holds because 
$$\phi_{M}^{\alpha_{M}}=\alpha_{M}^{-1}\co \phi_{M^{co H}\times H}\co (\alpha_{M}\ot H)$$
$$=
\phi_{M}\co (i_{M}\ot H)\co \nabla_{M}\co (M^{co H}\ot \mu_{H})\co ((\nabla_{M}\co (p_{M}\ot H)\co \rho_{M})\ot H)$$
$$=\phi_{M}\co (q_{M}\ot \mu_{H})\co (\rho_{M}\ot H).$$

On the other hand, by (\ref{action-alpha}), the coassociativity of $\delta_{H}$, (c1),  (c4) of Definition \ref{Hopf-module} and (\ref{new-c5}) we obtain:

\begin{itemize}
\item[ ]$\hspace{0.38cm} q_{M}^{\alpha_{M}}$

\item[ ]$=\phi_{M}\co (q_{M}\ot \mu_{H})\co (\rho_{M}\ot \lambda_{H}) \co \rho_{M}$

\item[ ]$= \phi_{M}\co (\phi_{M}\ot H)\co ((((\phi_{M}\co (M\ot \lambda_{H}))\ot H)\co (M\ot \delta_{H}))\ot \lambda_{H})\co (M\ot \delta_{H})\co \rho_{M} $

\item[ ]$=\phi_{M}\co ((\phi_{M}\co (M\ot \Pi_{H}^{R}))\ot \lambda_{H})\co (M\ot \delta_{H})\co \rho_{M} $

\item[ ]$=\phi_{M}\co ((\phi_{M}\co (M\ot \Pi_{H}^{R})\co \rho_{M})\ot \lambda_{H})\co \rho_{M} $

\item[ ]$= q_{M}.$
\end{itemize}

Then, $i_{M}=i_{M}^{\alpha_{M}}$ and $p_{M}=p_{M}^{\alpha_{M}}$ and, as a consequence, $\omega_{M}^{\alpha_{M}\prime}=(p_{M}\ot H)\co \rho_{M}=\omega_{M}^{\prime}.$ Moreover, by (c2-1) of Definition \ref{Hopf-module}, (\ref{tensor-idempotent-3}) and (\ref{nabla-i-phi})
\begin{itemize}
\item[ ]$\hspace{0.38cm}\omega_{M}^{\alpha_{M}} $

\item[ ]$=\phi_{M}^{\alpha_{M}}\co (i_{M}\ot H) $

\item[ ]$=\phi_{M}\co (q_{M}\ot \mu_{H})\co ((\rho_{M}\co i_{M}) \ot H)$

\item[ ]$=\phi_{M}\co (i_{M}\ot \mu_{H})\co (\nabla_{M}\ot  H)\co (M^{co H}\ot \eta_{H}\ot H)  $

\item[ ]$= \phi_{M}\co (i_{M}\ot H)\co \nabla_{M}$

\item[ ]$= \omega_{M}.$
\end{itemize}
 Therefore, $\nabla_{M}^{\alpha_{M}}=\nabla_{M}$ and then, for $(M, \phi_{M}^{\alpha_{M}}, \rho_{M})$, the associated isomorphism between $M$ and $M^{co H}\times H$ is $\alpha_{M}$.

Finally, by (\ref{tensor-idempotent-3})
$$(\phi_{M}^{\alpha_{M}})^{\alpha_{M}}=\phi_{M}\co (q_{M}\ot \mu_{H})\co ((\rho_{M}\co q_{M})\ot \mu_{H})\co 
(\rho_{M}\ot H)$$
$$=\phi_{M}\co (i_{M}\ot \mu_{H})\co ((\nabla_{M}\co (p_{M}\ot \eta_{H}))\ot \mu_{H})\co (\rho_{M}\ot H)$$
$$=\phi_{M}\co (i_{M}\ot H)\co\nabla_{M}\co (p_{M}\ot \mu_{H})\co (\rho_{M}\ot H)=\phi_{M}^{\alpha_{M}}$$
and (\ref{phi-square}) holds.
\end{proof}

\begin{remark}
Let $H$ be a weak Hopf quasigroup. The triple $(H, \phi_{H}=\mu_{H}, \rho_{H}=\delta_{H})$ is a right-right $H$-Hopf module and $\phi_{H}^{\alpha_{H}}=\phi_{H}$ because by (\ref{mu-assoc-2}), 
$$\phi_{H}^{\alpha_{H}}=\mu_{H}\co (\Pi_{H}^{L}\ot \mu_{H})\co (\delta_{H}\ot H)=\mu_{H}=\phi_{H}.$$

\end{remark}

\begin{definition}
\label{category}
Let $H$ be a weak Hopf quasigroup and let $(M, \phi_{M}, \rho_{M})$ and  $(N, \phi_{N}, \rho_{N})$ be right-right 
 $H$-Hopf modules. A morphism in ${\mathcal C}$ $f:M\rightarrow N$ is said to be $H$-quasilineal if the following identity holds
\begin{equation}
\label{quasilineal}
\phi_{N}^{\alpha_{N}}\co (f\ot H)=f\co \phi_{M}^{\alpha_{M}}.
\end{equation}
A morphism of right-right $H$-Hopf modules between $M$ and $N$ is a morphism $f:M\rightarrow N$ in ${\mathcal C}$ such that is both a morphism of right $H$-comodules and $H$-quasilineal. The collection of all right $H$-Hopf modules with their morphisms forms a category  which will be  denoted by ${\mathcal M}^{H}_{H}$. 

\end{definition}

\begin{proposition} 
\label{strong-hoof}
Let $H$ be a weak Hopf quasigroup and let $(M, \phi_{M}, \rho_{M})$ be an object in ${\mathcal M}^{H}_{H}$. Then, for   $(M^{co H}\times H, \phi_{M^{co H}\times H}, \rho_{M^{co H}\times H})$ the  identity 
\begin{equation}
\label{quasilineal-1}
\phi_{M^{co H}\times H}^{\alpha_{M^{co H}\times H}}=\phi_{M^{co H}\times H}
\end{equation}
holds.
\end{proposition}

\begin{proof} First note that, by (\ref{tensor-idempotent-2}) we have 
\begin{equation}
\label{new-q}
q_{M^{co H}\times H}=p_{M^{co H}\ot H}\circ (M^{co H}\ot  \Pi_{H}^{L})\co i_{M^{co H}\ot H}
\end{equation}
and, as a consequence, by (\ref{action-alpha}), the equality 
\begin{equation}
\label{new-q-2}
\phi_{M^{co H}\times H}^{\alpha_{M^{co H}\times H}}=p_{M^{co H}\ot H}\circ (M^{co H}\ot \mu_{H})\co 
((\nabla_{M}\co (M^{co H}\ot \Pi_{H}^{L}))\ot \mu_{H})\co (M^{co H}\ot \delta_{H}\ot H)\co (i_{M^{co H}\ot H}\ot H).
\end{equation}
holds.

Then, 

\begin{itemize}
\item[ ]$\hspace{0.38cm}\phi_{M^{co H}\times H}^{\alpha_{M^{co H}\times H}} $

\item[ ]$= p_{M^{co H}\ot H}\co (p_{M}\ot \mu_{H})\co ((\rho_{M}\co \phi_{M}\co (i_{M}\ot \Pi_{H}^{L}))\ot \mu_{H})\co (M^{co H}\ot \delta_{H}\ot H)\co  (i_{M^{co H}\ot H}\ot H)$

\item[ ]$= p_{M^{co H}\ot H}\co (p_{M}\ot \mu_{H})\co ((\rho_{M}\co \phi_{M}\co (\phi_{M}\ot \lambda_{H})\co (i_{M}\ot \delta_{H}))\ot \mu_{H})\co (M^{co H}\ot \delta_{H}\ot H)\co  (i_{M^{co H}\ot H}\ot H)$

\item[ ]$= p_{M^{co H}\ot H}\co (p_{M}\ot \mu_{H})\co ((\rho_{M}\co \phi_{M}\co (\phi_{M}\ot \lambda_{H}))\ot \mu_{H})\co (i_{M}\ot ((H\ot \delta_{H})\co \delta_{H})\ot H)\co  (i_{M^{co H}\ot H}\ot H)$

\item[ ]$= p_{M^{co H}\ot H}\co (p_{M}\ot \mu_{H})\co ((\rho_{M}\co \phi_{M}\co (M\ot \lambda_{H}))\ot \mu_{H})\co ( ((M\ot \delta_{H})\co \rho_{M}\co \phi_{M}\co (i_{M}\ot H)\co  i_{M^{co H}\ot H})\ot H)$

\item[ ]$= p_{M^{co H}\ot H}\co (p_{M}\ot \mu_{H})\co ((\rho_{M}\co q_{M})\ot \mu_{H})\co ( ( \rho_{M}\co \phi_{M}\co (i_{M}\ot H)\co  i_{M^{co H}\ot H})\ot H)$

\item[ ]$= p_{M^{co H}\ot H}\co (p_{M}\ot \mu_{H})\co ((\rho_{M}\co i_{M})\ot \mu_{H})\co ((\nabla_{M}\co i_{M^{co H}\ot H})\ot H)   $

\item[ ]$= p_{M^{co H}\ot H}\co (p_{M}\ot \mu_{H})\co ((\rho_{M}\co i_{M})\ot \mu_{H})\co ( i_{M^{co H}\ot H}\ot H)   $

\item[ ]$=p_{M^{co H}\ot H}\co (M^{co H}\ot \mu_{H})\co  ((\nabla_{M}\co (M^{co H}\ot \eta_{H}))\ot \mu_{H})\co ( i_{M^{co H}\ot H}\ot H) $

\item[ ]$=p_{M^{co H}\ot H}\co \nabla_{M}\co (M^{co H}\ot \mu_{H})\co  ( i_{M^{co H}\ot H}\ot H)  $

\item[ ]$=  \phi_{M^{co H}\times H}$

\end{itemize}

where the first equality follows by (\ref{new-q-2}) and the definition of $\nabla_{M}$, the second one by (c3) of Definition \ref{Hopf-module}, the third one by the coassociativity of $\delta_{H}$ and the fourth one by (\ref{new-c5-2-2}). In the fifth equality  we used  (c1) of Definition \ref{Hopf-module} and the sixth and eighth ones are consequence of the definition of $\nabla_{M}$. Finally, the seventh and the tenth one rely on the properties of $\nabla_{M}$ and the ninth one follows by 
(\ref{tensor-idempotent-3}).

\end{proof}

\begin{theorem} (Fundamental Theorem of Hopf modules)
Let $H$ be a weak Hopf quasigroup and let $(M, \phi_{M}, \rho_{M})$ be an object in ${\mathcal M}^{H}_{H}$. Then, the  right-right $H$-Hopf modules $(M, \phi_{M}, \rho_{M})$ and $(M^{co H}\times H, \phi_{M^{co H}\times H}, \rho_{M^{co H}\times H})$ are isomorphic in ${\mathcal M}^{H}_{H}$.

\end{theorem}

\begin{proof}
By Remark \ref{Th-mod} $\alpha_{M}=p_{M^{co H}\ot H}\circ \omega^{\prime}_{M}$ is an
isomorphism of right $H$-comodules with inverse $\alpha^{-1}_{M}=\omega_{M}\circ i_{M^{co H}\ot H}$.  To finish the proof we only need to show that (\ref{quasilineal}) holds.  Indeed, by (\ref{quasilineal-1}), (\ref{nabla-i-phi}) and (\ref{rho-p-phi}) we have 
$$\alpha_{M}^{-1}\co \phi_{M^{co H}\times H}^{\alpha_{M^{co H}\times H}}\co (\alpha_{M}\ot H)
=\phi_{M}\co 
(i_{M}\ot H)\co \nabla_{M}\co (M^{co H}\ot \mu_{H})\co ((\nabla_{M}\co (p_{M}\ot H)\co \rho_{M})\ot H)=\phi_{M}^{\alpha_{M}}.$$
\end{proof}

\section*{Acknowledgements}
The  authors were supported by  Ministerio de Econom\'{\i}a y Competitividad and by Feder founds. Project MTM2013-43687-P: Homolog\'{\i}a, homotop\'{\i}a e invariantes categ\'oricos en grupos y \'algebras no asociativas.


\begin{thebibliography}{99}

\bibitem{AFG1} J.N. Alonso \'Alvarez, J.M. Fern\'andez Vilaboa,
R. Gonz\'alez Rodr\'{\i}guez, \textit{Weak Hopf algebras and weak
Yang-Baxter operators}, J. Algebra,  \textbf{ 320} (2008), 2101-2143.

\bibitem{IND}
 J. N. Alonso \'Alvarez,  J.M. Fern\'{a}ndez Vilaboa,
R. Gonz\'{a}lez Rodr\'{\i}guez,  \textit{Weak braided Hopf algebras}, 
Indiana Univ. Math. J.   \textbf{57} (2008), 2423-2458.

\bibitem{BEN} J. B\'enabou, \textit{Introduction to bicategories}, in Reports of the Midwest Categorical Seminar,  LNM \textbf{47}, Springer, 1967, pp. 1-77.

\bibitem{bohm} G. B\"{o}hm, G., F. Nill, K. Szlach\'anyi,
 \textit{Weak Hopf algebras, I. Integral theory and
$C^{\ast}$-structure}, J.  Algebra, \textbf{221}  (1999), 385-438.

\bibitem{Brz}
T. Brzezi\'nski, \textit{Hopf modules and the fundamental theorem for Hopf (co)quasigroups}, Internat. Elec. J. Algebra,  \textbf{8} (2010), 114-128.

\bibitem{ENO} P. Etingof, D., Nikshych, V., Ostrick, \textit{On fusion
categories}, Ann. Math., \textbf{162} (2005), 581-642.


\bibitem{Ha} T. Hayashi,  \textit{ Face algebras I. A generalization of
quantum group theory}, J. Math. Soc. Japan, \textbf{ 50} (1998),
293-315.

 \bibitem{JS} A. Joyal, R. Street, \textit{Braided tensor categories}, Adv.
 Math., \textbf{102}  (1993), 20-78.

\bibitem{Karoubi}
 M. Karoubi, \textit{K-th\'{e}orie},  Les Presses de
l'Universit\'{e} de Montr\'{e}al, Montr\'{e}al, 1971.

\bibitem{Mac} S. Mac Lane, \textit{Categories for the Working Mathematician}, GTM \textbf{5}, Springer-Verlag, New-York, 1971.
 
\bibitem{Majidesfera}
J. Klim, S. Majid, \textit{Hopf quasigroups and the algebraic 7-sphere},
J. Algebra \textbf{323} (2010), 3067-3110.



\bibitem{NV} {\rm D. Nikshych, L. Vainerman,}
\textit{ Finite quantum groupoids and their applications},  New
Directions in Hopf Algebras, MSRI Publications {\bf 43} (2002),
211-262.

\bibitem{PA3} B. Pareigis,  \textit{On braiding and Dyslexia},
  J. Algebra,  \textbf{171} (1995), 413-425.

\bibitem{PIS}
J.M. P\'erez-Izquierdo, I.P. Shestakov, \textit{An envelope for Malcev algebras},
J. Algebra \textbf{272} (2004), 379-393.

\bibitem{PI2}
J.M. P\'erez-Izquierdo, \textit{Algebras, hyperalgebras, nonassociative
bialgebras and loops},
 Adv. Math. \textbf{208} (2007), 834-876.
 

\bibitem{EmilioPura}
M.P. L\'opez, E. Villanueva,  \textit{The antipode and the (co)invariants of
a finite Hopf (co)quasigroup},  Appl. Cat. Struct. \textbf{21}
(2013), 237-247.



\bibitem{Ya} Y. Yamanouchi, \textit{Duality for generalized Kac algebras
and characterization of finite groupoid algebras}, J. Algebra 
\textbf{163}, (1994), 9-50.





\end{thebibliography}
\end{document}